\documentclass{article}
\usepackage[utf8]{inputenc}
\usepackage{amsthm}
\usepackage{graphicx}
\usepackage{amssymb}
\usepackage{color}
\usepackage[dvipsnames]{xcolor}
\usepackage{tikz}
\usepackage{tikz-cd}
\usetikzlibrary{shapes.geometric}
\usepackage{tcolorbox}
\usepackage{amsmath}
\usepackage[hidelinks]{hyperref}
\usepackage{float}
\usepackage{mathtools}
\usepackage{enumitem}
\usepackage{geometry}
\theoremstyle{plain}
\newtheorem{theorem}{Theorem}[section]
\newtheorem{lemma}[theorem]{Lemma}
\newtheorem{corollary}[theorem]{Corollary}

\newtheorem{proposition}[theorem]{Proposition}
\newtheorem{claim}[theorem]{Claim}
\newtheorem{problem}[theorem]{Problem}

\theoremstyle{definition}
\newtheorem{definition}[theorem]{Definition}
\newtheorem{exam}[theorem]{Example}

\theoremstyle{remark}
\newtheorem{remark}[theorem]{Remark}
\newtheorem*{remark*}{Remark}

\newcommand{\RR}{\mathbb{R}}

\newcommand{\CC}{{\mathbb C}}

\newcommand{\ical}{\mathcal{I}}
\newcommand{\lcal}{\mathcal{L}}

\newcommand{\scal}{\mathcal{S}}

\newcommand{\vcal}{\mathcal{V}}\newcommand{\ucal}{\mathcal{U}}
\newcommand{\zcal}{\mathcal{Z}}

\newcommand{\vol}{{\operatorname{Vol}}}

\def\Im{{\operatorname{Im}}}
\def\Re{{\operatorname{Re}}}

\def\Diff{{\operatorname{Diff}}}

\def\id{{\operatorname{id}}}

\def\path{\operatorname{Path}}

\def\Imm{\hbox{\rm Imm}}

\usepackage[style=numeric, giveninits=true, sortcites=true, backend=biber]{biblatex} 
\AtEveryBibitem{\clearfield{isbn}}
\AtEveryBibitem{\clearfield{issn}}
\AtEveryBibitem{\clearfield{doi}}
\AtEveryBibitem{\clearfield{url}}
\AtEveryBibitem{\clearfield{number}}

\DeclareNameAlias{default}{given-family}

\DefineBibliographyExtras{english}{%
}

\renewbibmacro{in:}{
}

\DeclareFieldFormat*{title}{\mkbibemph{#1}}

\renewbibmacro*{title}{%
  \printfield{title}%
  \ifentrytype{misc}{\addcomma\addspace preprint}{}%
}

\DeclareFieldFormat{pages}{#1}

\DeclareFieldFormat*{journaltitle}{\textnormal{#1}}

\DeclareFieldFormat*{volume}{\mkbibbold{#1}}

\usepackage{minitoc}

\usepackage[toc,page]{appendix}

\title{Moduli of special Lagrangians with boundary, I: Unobstructed Deformations}
\author{Vasanth Pidaparthy\thanks{The author would like to thank Y. A. Rubinstein for suggesting this problem and for his constant encouragement and support. Research supported in part by NSF grants DMS-1906370, 2204347, and BSF grants 2016173, 2020329.}}
\date{\today}
\begin{document}

\maketitle

\begin{abstract}
    This article studies the deformation problem for compact special Lagrangians with boundary in a Calabi--Yau manifold, with each boundary component constrained along a given Lagrangian submanifold. The tangent vectors generating such deformations are identified with harmonic 1-forms vanishing on the boundary of the special Lagrangian, and the deformation generated by any such tangent vector is unobstructed. Consequently, the moduli space of special Lagrangians with boundary is a smooth manifold whose dimension equals the dimension of the first relative cohomology of the special Lagrangian.
\end{abstract}

\section{Introduction}

This is the first of two articles that studies the moduli space of special Lagrangians with boundary (SLb) in a Calabi--Yau manifold $(X, \omega, J, \Omega)$, whose boundary is constrained along a given union of Lagrangian submanifolds (Definition~\ref{def:HWB-boundary-conditions-on-immersions-and-lagrangians}) \[\Lambda_1,\dots, \Lambda_d \subset X.\]

The goal of the present article is to characterize the tangent vectors generating paths of special Lagrangians with boundary. We show that all such tangent vectors generate paths of special Lagrangians with boundary, making the deformation problem unobstructed. Consequently, special Lagrangians with boundary form smooth finite-dimensional moduli spaces.

Special Lagrangian submanifolds were introduced by Harvey--Lawson as an example of minimal submanifold in a calibrated manifold \cite[{\S}III]{HarLaw-calibratedgeom}. Any smooth Lagrangian minimal submanifold is special Lagrangian \cite[Proposition III.2.17]{HarLaw-calibratedgeom}, and a closed special Lagrangians minimizes volume in its homology class \cite[Lemma II.3.5, Theorem III.1.10]{HarLaw-calibratedgeom}. Special Lagrangians have attracted considerable interest both as minimal submanifolds and canonical representatives of homology classes. Existence and moduli spaces of closed special Lagrangian submanifolds and special Lagrangian torii in a Calabi--Yau manifold have been widely studied in mirror symmetry and are the subject of major conjectures such as the Thomas--Yau conjectures on existence of special Lagrangians \cites{thomas-MR1882337, Thomas-Yau-MR1957663, Joyce-TY-conj-MR3354954, YangLi2022thomasyau, Solomon2013TheCH}, and the Strominger--Yau--Zaslow conjecture on the structure of the moduli space of Lagrangian torii \cite{SYZ-mirror-Tduality, YangLi_2022-syz-conjecture}.

A significant challenge in studying special Lagrangians is the scarcity of examples. Many known examples assume the Calabi--Yau has a large group of symmetries under which the special Lagrangian is invariant \cite{joyce-slag-m-folds-MR1932324, Solomon-yuval-MR3658152}. On the other hand, special Lagrangians with boundary are more likely to exist in Calabi--Yau manifolds. For example, $\CC^n$ admits no closed special Lagrangians since it admits no closed minimal submanifolds, but $\CC^n$ admits plenty of special Lagrangian gradient graphs with boundary \cite[Theorem 5]{caffarelli-nirenberg-spruck-dirichlet-problem-III-MR0806416}. With this motivation in mind, the goal of this article and its sequel is to study the moduli space of special Lagrangians with boundary. Two basic questions about special Lagrangians with boundary are:

\begin{problem}
\label{problem:HWB-1}
    What are natural boundary conditions to apply on special Lagrangians?
\end{problem}

\begin{problem}
\label{problem:HWB-2}
    Under what conditions does a Calabi--Yau manifold admit a special Lagrangian submanifold with boundary?
\end{problem}

Problem~\ref{problem:HWB-1} is open ended and has multiple proposed solutions. Problem~\ref{problem:HWB-1} goes back to the work of Butscher who studied special Lagrangians whose boundary is contained in a codimension two symplectic submanifold \cite{Butscher2002}. Solomon--Yuval revisited Problem~\ref{problem:HWB-1} and introduced a different family of Lagrangian boundary conditions \cite[Notation~2.7]{SolomonYuval2020-geodesics}, and this is the notion we will study in this article (Definition~\ref{def:HWB-boundary-conditions-on-immersions-and-lagrangians}).

Their main focus was on compact special Lagrangians \textit{cylinders}---manifolds of the form $N\times [0,1]$ for $N$ compact and closed---whose two boundary components are constrained to lie on a pair of given Lagrangian submanifold of $X$ (Definition \ref{def:HWB-boundary-conditions-on-immersions-and-lagrangians}). Their motivation came from a correspondence that they discovered between the moduli space of immersed special Lagrangian cylinders with their Lagrangian boundary conditions, and geodesics of closed embedded positive Lagrangians in $X$ \cite[Theorem 1.5]{SolomonYuval2020-geodesics}. They called this correspondence the \textit{cylindrical transform} and used it to prove openness of $C^{1,\alpha}$ geodesics \cite[Theorem 1.6]{SolomonYuval2020-geodesics}. The cylindrical transform also produced examples of moduli spaces of special Lagrangian cylinders of topological type $[0,1] \times S^{n-1}$ in Milnor fibers \cite{Solomon-yuval-MR3658152, SolomonYuval2020-geodesics}.

The solution to Problem~\ref{problem:HWB-2} is expected to depend on some form of algebro-geometric stability such as Bridgeland stability and its variants \cite{Thomas-Yau-MR1957663, bridgeland-stab-MR2373143, Joyce-TY-conj-MR3354954, YangLi2022thomasyau}. This article does not discuss Problem~\ref{problem:HWB-2}, but instead focuses on the properties of the moduli space of special Lagrangians with boundary, assuming they exist.

\begin{problem}
\label{problem:HWB-3.1}
    Is the deformation problem for special Lagrangians with boundary unobstructed? What are the generators of deformations of special Lagrangians with boundary? Is the moduli space finite-dimensional?
\end{problem}

Motivated by Harvey--Lawson's work on calibrated geometries \cite{HarLaw-calibratedgeom}, McLean studied deformations of calibrated submanifolds in general and special Lagrangians in particular. Problem~\ref{problem:HWB-3.1} was settled for \textit{closed} special Lagrangians \cite[{\S}3]{Mclean1998}.

\begin{theorem}
    {\rm\cite[Theorem~3.6]{Mclean1998}} Let $(X, \omega, J, \Omega)$ be a Calabi-Yau manifold and $L \subset X$ a compact special Lagrangian submanifold without boundary. Then the set of special Lagrangian submanifolds sufficiently close to $L$ is a finite-dimensional manifold parametrized by the set of harmonic 1-forms on $L$.
\end{theorem}

Butscher studied Problem~\ref{problem:HWB-3.1} for compact embedded special Lagrangians with boundary $L\subset X$ whose boundary is constrained along a given codimension two symplectic submanifold. He proved the moduli space is a finite-dimensional manifold parametrized by harmonic 1-forms on $L$ satisfying a Neumann boundary condition. Solomon--Yuval studied Problem~\ref{problem:HWB-3.1} for immersed special Lagrangian cylinders satisfying their Lagrangian boundary conditions, i.e., the case $d=2$ with $L=[0,1]\times N$ \cite[Proposition 4.7]{SolomonYuval2020-geodesics}. In this case, the moduli space is 1-dimensional and is parametrized by harmonic 1-forms satisfying Dirichlet boundary conditions.

\begin{theorem}
\label{thm:HWB-solomon-yuval-theorem}
    {\rm\cite[Proposition 4.7]{SolomonYuval2020-geodesics}} Let $\Lambda_0, \Lambda_1$ be Lagrangian submanifolds of a Calabi--Yau manifold $(X,\omega, J, \Omega)$. Suppose $f: N\times[0,1] \rightarrow X$ is a special Lagrangian immersion for a closed connected smooth $(n-1)$-manifold $N$ satisfying the boundary conditions in Definition~\ref{def:HWB-boundary-conditions-on-immersions-and-lagrangians}.1. Then the set of special Lagrangian immersions sufficiently close to $f$ (up to reparametrization) form a 1-dimensional manifold locally parametrized by harmonic 1-forms on $N\times [0,1]$ vanishing on the boundary.
\end{theorem}

\section{Results}
\label{section:HWB-results}

This article resolves Problem~\ref{problem:HWB-3.1} in full generality for SLb satisfying the Lagrangian boundary conditions introduced by Solomon--Yuval (Definition~\ref{def:HWB-boundary-conditions-on-immersions-and-lagrangians}). Inspired by Solomon--Yuval \cite{SolomonYuval2020-geodesics}, Problem~\ref{problem:HWB-3.1} is studied for \textit{immersed} SLb. Problem~\ref{problem:HWB-3.1} is answered by Proposition~\ref{prop:HWB-tangentspacetospeciallagrangianswithboundary-1} and Theorem~\ref{thm:HWB-tangentspacetospeciallagrangianswithboundary-2}.

Let $(X,\omega, J, \Omega)$ be a Calabi--Yau manifold (Definition~\ref{def:HWB-calabi-Yau-manifold}) of complex dimension $n$ and let $L$ be an $n$-dimensional real compact manifold with smooth boundary and $d$ boundary components $C_1,\dots, C_d$. Fix disjoint embedded Lagrangian submanifolds $\Lambda_1,\dots, \Lambda_d \subset X$. The moduli space of immersed special Lagrangian submanifolds of $X$ diffeomorphic to $L$ with the boundary component $C_i$ lying on $\Lambda_i$ is denoted by (Definition~\ref{def:HWB-boundary-conditions-on-immersions-and-lagrangians})
\[\scal\lcal := \scal\lcal(X,L; \Lambda_1,\dots, \Lambda_d).\]

First we observe that the tangent 1-form associated to a path of special Lagrangians in $\scal\lcal$ (Lemma~\ref{lemma:HWB-formula-for-derivative-of-path-of-Lagrangians}) is harmonic and vanishes on the boundary.

\begin{proposition}
\label{prop:HWB-tangentspacetospeciallagrangianswithboundary-1}
Fix a compact manifold $L$ with boundary components $C_1,\dots, C_d$ and disjoint Lagrangian submanifold $\Lambda_1,\dots, \Lambda_d \subset X$. Consider a smooth path of special Lagrangian immersions
\[f: [0,1] \times L \rightarrow X, \quad f\in C^\infty,\quad f_t:=f(t,\cdot),\]
satisfying boundary conditions $\Lambda_1,\dots, \Lambda_d$ (Definition~\ref{def:HWB-boundary-conditions-on-immersions-and-lagrangians}). Let $\star_t $ denote the Hodge star for the pullback Riemannian metric $g_t = f^*g$ on $L$, and let
\begin{equation}
\label{eq:HWB-thm-1-theta-phi-notation}
    \begin{aligned}
    \theta_t:=&\ f_t^* \left( i_{\frac{df_t}{dt}}\omega \right) \\ 
    \phi_t :=&\ f_t^* \left( i_{\frac{df_t}{dt}} \Im(\Omega) \right) . 
\end{aligned} \qquad t\in [0,1],
\end{equation}
where $i_V\omega := \omega(V,\cdot)$ is the interior derivative. Then the following hold for each $t\in[0,1]$.
\begin{enumerate}[label={\arabic*.}]
    \item $\theta_t$ and $\phi_t$ are closed differential forms on $L$, and $\theta_t|_{\partial L} \equiv 0$.
    \item $ \star_t \theta_t = \phi_t$.
\end{enumerate}
In particular $\theta_t$ and $\phi_t$ are closed and co-closed with respect to the pullback metric $f_t^* g$, and are harmonic. Furthermore $\phi_t$ need not vanish on $\partial L$ and in fact satisfies Neumann boundary conditions by Lemma~\ref{def:boundarydecompositionofformsontheboundary}.
\end{proposition}

Problem~\ref{problem:HWB-3.1} is resolved by a careful application of the implicit function theorem and Proposition~\ref{prop:HWB-tangentspacetospeciallagrangianswithboundary-1}:

\begin{theorem}
\label{thm:HWB-tangentspacetospeciallagrangianswithboundary-2}
    Under the assumptions of Proposition~\ref{prop:HWB-tangentspacetospeciallagrangianswithboundary-1}, $\scal \lcal(X,L; \Lambda_1,\dots, \Lambda_d)$ is a finite-dimensional manifold whose tangent space at an immersed special Lagrangian is isomorphic to the space of closed and co-closed 1-forms on $L$ vanishing at the boundary. In particular $\scal \lcal$ is a manifold satisfying $\dim \scal\lcal = \dim H^1(L,\partial L; \RR) = \dim H^{n-1}(L; \RR)$.
\end{theorem}

The proofs of Proposition~\ref{prop:HWB-tangentspacetospeciallagrangianswithboundary-1} and Theorem~\ref{thm:HWB-tangentspacetospeciallagrangianswithboundary-2} combine the ideas of Mclean \cite[{\S}3]{Mclean1998} and Solomon--Yuval \cite[{\S}4.1]{SolomonYuval2020-geodesics}. Solomon--Yuval considered the special case when $L = [0,1]\times N$ is a cylinder and $d=2$. In this case, $H^1(L,\partial L;\RR) \cong \RR$ and closed 1-forms vanishing on the boundary are exact \cite[Lemma~4.3]{SolomonYuval2020-geodesics}, and Theorem~\ref{thm:HWB-tangentspacetospeciallagrangianswithboundary-2} results from an implicit function theorem argument for the Laplace operator on functions. When $L$ is not a cylinder, the argument proceeds similarly by applying the implicit function theorem for the Laplace--Beltrami operator on 1-forms. 

McLean's argument dealt only with embedded manifolds. Most of the argument generalizes easily from embeddings to immersions. The only subtle point is that a parametrization of an open subset of immersions with Lagrangian boundary conditions needs more care than for embeddings. The argument is a generalization of Cervera--Mascar\'o--Michor \cite[Theorem 1.5]{Cervera-Mascaro-MichorMR1244452} to manifolds with boundary. The generalized version is briefly discussed in Solomon--Yuval \cite[Corollary~2.10]{SolomonYuval2020-geodesics}.  The complete argument is provided in Corollary~\ref{cor:HWB-boundary-lagrangian-deformation-space} for the reader's reference.

\textbf{Organization.} Section~\ref{sec:HWB-preliminaries} sets up the notation and definitions related the moduli of immersed Lagrangians and special Lagrangians following Akveld--Salamon \cite{Akveld2001} and Solomon--Yuval \cite{SolomonYuval2020-geodesics}. Hodge theory on manifolds with boundary is also recalled in this section. Section~\ref{sec:HWB-moduli-space-of-special-lagrangians} studies the elliptic PDE governing SLb, proving Proposition~\ref{prop:HWB-tangentspacetospeciallagrangianswithboundary-1} and Theorem~\ref{thm:HWB-tangentspacetospeciallagrangianswithboundary-2}. Section~\ref{sec:HWB-generalizations} briefly discusses generalizations to almost Calabi--Yau manifolds.

\section{Preliminaries}
\label{sec:HWB-preliminaries}

Section~\ref{sec:modulispaceoflagrangianswithboundary} recalls the symplectic geometry of immersed submanifolds with boundary in a Calabi--Yau manifold, following the conventions in Solomon--Yuval \cite[{\S}2]{SolomonYuval2020-geodesics}. The Lagrangian boundary conditions defining the moduli space of immersed Lagrangians and special Lagrangians are recalled in Definitions \ref{def:HWB-boundary-conditions-on-immersions-and-lagrangians} and \ref{def:HWB-moduli-space-of-special-lagrangians}. Section~\ref{subsec:HWB-weinstein-neighbrohood-with-boundary} recalls a parametrization of immersions in a small neighborhood of a given free Lagrangian immersion in terms of 1-forms on the Lagrangian \cite[Lemma 2.8]{SolomonYuval2020-geodesics}. The parametrization is constructed in Corollary~\ref{cor:HWB-boundary-lagrangian-deformation-space} and is the domain of the elliptic special Lagrangian operator. Section~\ref{subsec:HWB-smooth-paths-and-tangents} recalls the notion of a smooth path of immersed SLb, and the associated tangent 1-form introduced by Akveld--Salamon \cite[{\S}2]{Akveld2001}. Section~\ref{subsec:hodgedecomposition} recalls Hodge theory for manifolds with boundary.

\subsection{Special Lagrangians and Lagrangian boundary conditions}
\label{sec:modulispaceoflagrangianswithboundary}

\begin{definition}
    \label{def:HWB-calabi-Yau-manifold}
    A quadruple $(X,\omega,J,\Omega)$ is called a \textit{Calabi--Yau manifold} if $(X,J)$ is a K\"ahler manifold equipped with a Ricci flat K\"ahler metric $\omega$ and a holomorphic top form $\Omega$ of constant length. Since $\Omega$ has constant length it satisfies
\begin{equation}
    (-1)^{\frac{n(n-1)}{2}} \left( \frac{\sqrt{-1}}{2}\right)^n \Omega \wedge \overline{\Omega} = \frac{1}{n!} \omega^n.
    \label{eq:HWB-calabi-Yau-equation}
\end{equation}
The Riemannian metric on $X$ is denoted by
\begin{equation}
    \label{eq:metriconakahlermanifold}
    g := \omega(\cdot, J \cdot).
\end{equation}
\end{definition}

In this article a Calabi--Yau manifold is assumed to be Ricci flat. A K\"ahler manifold with trivial canonical bundle which is not Ricci flat is called \textit{almost Calabi--Yau}. Special Lagrangians are not minimal submanifolds in an almost Calabi--Yau manifold, but instead are mimal submanifolds for a Riemannian metric conformal to \eqref{eq:metriconakahlermanifold}. The results of this article also generalize to almost Calabi--Yau manifolds (see Section~\ref{sec:HWB-generalizations}).

\begin{definition}~
\label{def:immersionsandimmersedforms}
\begin{enumerate}[label={(\arabic*)}]
    \item Let $L$ be a smooth compact oriented manifold of dimension $n$ with smooth boundary. $\Diff(L)$ denotes the group of diffeomorphisms of $L$ that preserve boundary components of $L$, i.e., $\psi \in \Diff(L)$ if
    \begin{enumerate}[label={(\roman*)}]
        \item $\psi:L \rightarrow L$ is a diffeomorphism and
        \item $\psi(C) \subset C$ for every connected component of the boundary $C\subset \partial L$.
    \end{enumerate}

    \item An \textit{immersion} $f:L \rightarrow X$ is a smooth map for which the derivative $df_p: T_pL \rightarrow T_{f(p)}X$ is injective for every $p\in L$. An immersion $f$ is called \textit{free} if $f\circ \psi = f$ and $\psi \in \Diff(L)$ implies $\psi = \id_L$, i.e., $f$ has trivial stabilizer under the reparametrization action of $\Diff(L)$. An injective immersion is always free, and when $L$ is compact an injective immersion is also an embedding.
    
    \item An \textit{immersed submanifold of $X$ of type $L$} is an equivalence class of immersions $f:L \rightarrow X$ under the reparametrization action of $\Diff(L)$,
    \[f \sim f' \quad \iff \quad  \exists \ \psi\in \Diff(L)\quad s.t. \quad f' = f\circ \psi. \]
    The equivalence class of an immersion $f$ is denoted by $[f]$.
    
    \item An immersed submanifold is called \textit{free} if it admits a representative that is a free immersion. It is immediate that every representative is free.
\end{enumerate}
\end{definition}

The motivation for working with free immersions is to avoid immersions with non-trivial stabilizers under the reparametrization action of $\Diff(L)$. The typical example of a non-free immersion is a map that factors through a covering map.

\begin{exam}
    Suppose $p:L \rightarrow B$ is a smooth covering map and $g:B \rightarrow X$ is an embedding, then $f:= g\circ p:L \rightarrow X$. If $p$ is not a diffeomorphism, then $f$ fails to be a free immersion since $f$ is stabilized by the deck transformations of $p:L \rightarrow B$. 
\end{exam}
The typical way to check that an immersion is free is to construct an embedded point.

\begin{lemma}
    Let $f:L \rightarrow X$ be an immersion and suppose $p\in L$ is an embedded point, i.e., $f^{-1}(f(p)) = p$. Then $f$ is a free immersion.
\end{lemma}
\begin{proof}
    Refer to Cervera--Mascar\'o--Michor \cite[Lemma~1.4]{Cervera-Mascaro-MichorMR1244452} and Solomon--Yuval \cite[Lemma 2.4, 2.5]{SolomonYuval2020-geodesics}.
\end{proof}

The results of this article are formulated in terms of immersed special Lagrangians. However most calculations proceed by fixing a specific parametrization or representative immersion, and the distinction between immersions and embeddings does not play a serious role in this article.

\begin{definition}~
\label{def:immersedlagrangiansandspeciallagrangians}
\begin{enumerate}[label={(\arabic*)}]
    \item An immersion $f:L \rightarrow X$ is \textit{Lagrangian} if $f^*\omega = 0$. Note that being Lagrangian is $\Diff(L)$ invariant.

    \item A Lagrangian immersion $f:L \rightarrow X$ is called \textit{special Lagrangian} if in addition $f^*\Im(\Omega) = 0$.
    
    \item An immersed manifold $[f]$ of type $L$ in $X$ is (special) Lagrangian if one representative is (special) Lagrangian. Note that if one representative is (special) Lagrandian then all representatives are so.
\end{enumerate}
\end{definition}

Definition \ref{def:HWB-boundary-conditions-on-immersions-and-lagrangians} recalls the Lagrangian boundary conditions for Lagrangians with boundary, introduced by Solomon--Yuval \cite[Notation 2.7]{SolomonYuval2020-geodesics}.

\begin{definition}
\label{def:HWB-boundary-conditions-on-immersions-and-lagrangians}
Let $L$ be a compact manifold with smooth boundary, and let $C_1,\dots, C_d$ denote its boundary components. Fix an ordering $C_1,\dots, C_d$ on the boundary components. Fix pairwise disjoint embedded Lagrangian submanifolds $\Lambda_1, \dots, \Lambda_d \subset X$.
\begin{enumerate}
    \item A smooth map $f:L \rightarrow X$ is said to have \textit{boundary condition $\Lambda_1,\dots, \Lambda_d$} if
\begin{enumerate}[label={(\alph*)}]
    \item $f(C_i) \subset \Lambda_i$ for $i=1,\dots, d$, and
    \item $df(T_pL) \not\subset T_{f(p)} \Lambda_i$ for all $p\in C_i$ and $i=1,\dots, d$.
\end{enumerate}
    The set of smooth maps with boundary conditions $\Lambda_1,\dots, \Lambda_d$ is denoted by $C^\infty(L, X; \Lambda_1,\dots, \Lambda_d)$. When 

    If $f$ is in addition a free immersion, then $f|_{C_i}: C_i \rightarrow \Lambda_i$ is an immersion for $i=1,\dots, d$. Note that both (a) and (b) are invariant under reparametrization by $\Diff(L)$ because by Definition \ref{def:immersionsandimmersedforms}(1) elements of $\Diff(L)$ preserve boundary components.

\item An immersed submanifold of type $L$ is said to have boundary conditions $\Lambda_1, \dots, \Lambda_d$ if one and hence every representative immersion has these boundary condition. 
\end{enumerate}
Implicit in Definition~\ref{def:HWB-boundary-conditions-on-immersions-and-lagrangians}.1 is an ordering of the $d$ boundary components of $L$ corresponding to the ordering $\Lambda_1,\dots, \Lambda_d$.
\end{definition}

\begin{remark}
    There have been other notions of boundary conditions in the study of special Lagrangians. Butscher studied a different class of special Lagrangians whose boundary is constrained to lie on a fixed co-dimension 2 symplectic submanifold \cite[p.~3]{Butscher2002}, and generalized McLean's deformation theorem under these boundary conditions \cite{Mclean1998}. The resulting moduli space of special Lagrangians is fairly different from the one studied in this article, and its tangent space is identified with harmonic 1-forms satisfying a Neumann boundary condition (compare with Theorem~\ref{thm:HWB-tangentspacetospeciallagrangianswithboundary-2}).
\end{remark}
 
\begin{definition}
\label{def:HWB-moduli-space-of-special-lagrangians}
Let $L$ have $d$ boundary components $C_1,\dots, C_d$, and let $\Lambda_1, \dots, \Lambda_d \subset X$ be pairwise disjoint embedded Lagrangians.
\begin{enumerate}[label={(\arabic*)}]
    \item The moduli space of smooth free immersed Lagrangian submanifolds of $X$ of type $L$ with boundary conditions $\Lambda_1,\dots, \Lambda_d$ is denoted by
    \[\lcal := \lcal(X,L; \Lambda) := \lcal (X,L ; \Lambda_1,\dots, \Lambda_d).\]
    \item The moduli space of smooth free immersed special Lagrangian submanifolds in $X$ of type $L$ with boundary conditions $\Lambda_1,\dots, \Lambda_d$  is denoted by
    \begin{gather*}
        \scal\lcal :=  \scal \lcal(X,L; \Lambda) :=\scal \lcal (X,L ; \Lambda_1,\dots, \Lambda_d).
    \end{gather*}
\end{enumerate}
As in Definition~\ref{def:HWB-boundary-conditions-on-immersions-and-lagrangians}, the boundary components $C_1,\dots, C_d$ of $L$ and the Lagrangians $\Lambda_1,\dots, \Lambda_d$ are assumed to be ordered. The boundary data $\Lambda_1,\dots, \Lambda_d$ is often abbreviated to $\Lambda$ or omitted to simplify notation.
\end{definition}

A key observation of Harvey--Lawson is that the special Lagrangian condition $f^*\Im(\Omega) \equiv 0$ is equivalent to the calibration condition $f^*\Re(\Omega) = \vol_{f^* g}$.

\begin{lemma}
\label{lemma:HWB-harveylawson-calibration-lemma}
    {\rm\cite[Theorem~III.1.10]{HarLaw-calibratedgeom}} If $f:L \rightarrow X$ is a special Lagrangian immersion then $L$ is calibrated by $\Re(\Omega)$, i.e.,
    \[f^*\Re(\Omega) = \vol_{f^*g},\]
    where $g = \omega(\cdot, J\cdot)$ is the Riemannian metric on $X$. Furthermore $f:L \rightarrow X$ is an immersed minimal submanifold of $X$.
\end{lemma}
Lemma~\ref{lemma:HWB-harveylawson-calibration-lemma} is only used in the proof of Lemma~\ref{lemma:HWB-phitishodgestarofthetat} to relate the special Lagrangian condition to the Riemannian Hodge star operator.

\subsection{Deformations with Lagrangian boundary conditions}
\label{subsec:HWB-weinstein-neighbrohood-with-boundary}

The moduli spaces of special Lagrangians $\scal \lcal(X,L,\Lambda_1,\dots, \Lambda_d)$ can be identified as the zero set of the special Lagrangian operator, an elliptic partial differential operator. The domain of this operator is an open neighborhood in $C^\infty(L,X;\Lambda_1,\dots, \Lambda_d)$  (Definition~\ref{def:HWB-boundary-conditions-on-immersions-and-lagrangians}). The space of immersions with boundary can be parametrized in the neighborhood of a Lagrangian immersion by an open set of 1-forms \cite[Lemma~2.8, Corollary~2.10]{SolomonYuval2020-geodesics}. This section is concerned with a description of this parametrization (Corollary~\ref{cor:HWB-boundary-lagrangian-deformation-space}). The results of this section are from Solomon--Yuval \cite[{\S}2]{SolomonYuval2020-geodesics}.

Fix a Lagrangian immersion $f_0 \in C^\infty(L, X; \Lambda_1,\dots, \Lambda_d)$ (Definition~\ref{def:HWB-boundary-conditions-on-immersions-and-lagrangians}). Recall that deformations of $f_0$ as an immersion are generated by sections of the pullback normal bundle,
\[f_0^*NL \rightarrow L, \qquad (f_0^*NL)_p = (NL)_{f(p)} = (T_pL)^{\perp}.\]
Given a section of the pullback normal bundle, its exponential flow generates a deformation of the immersion $f_0$. Furthermore if $f_0$ were in fact a Lagrangian immersion,then the contraction $V\mapsto i_V\omega$ defines an isomorphism $f_0^*NL \cong T^*L $, and deformations of $f_0:L \rightarrow X$ can be identified with 1-forms on $L$. Unfortunately, these natural deformations of $f_0$ need not respect the Lagrangian boundary conditions $\Lambda_1,\dots, \Lambda_d$. There are two ways this can go wrong.
\begin{enumerate}
    \item The normal vector field need not be tangent to the boundary condition $\Lambda_i$ over the boundary component $C_i$ (Figure~\ref{fig:HWB-normal-vector-field-vs-boundary-cond}).
    \item Even if the normal vector field is tangent to the boundary conditions $\Lambda_i$ over $C_i$, the geodesic along this direction might exit $\Lambda_i$ instantaneously.
\end{enumerate}

The second issue can be remedied easily by replacing the replacing Riemannian exponential flow with the exponential flow of an affine connection $\nabla^\Lambda$ for which $\Lambda_1,\dots, \Lambda_d$ are all geodesic submanifold (Proposition~\ref{prop:HWB-boundary-Lagrangian-vector-bundle-connection-2}).

To resolve  the first issue, the pullback normal bundle $NL \subset f^*TX$ is replaced by a diffrenet vector bundle of the same rank. When $f:L \rightarrow X$ is a Lagrangian embedding recall that $NL = Jdf(TL)$. The idea is to choose a different complex rotation $E = (a+J)df(TL)$ such that $E\cap T\Lambda_i \neq \{0\} $ over $C_i$, $1\le i\le d$. Then the sections of $E$ which are tangent to $\Lambda_1,\dots, \Lambda_d$ over the boundary will generate the deformations of interest (Proposition~\ref{prop:HWB-boundary-Lagrangian-vector-bundle-connection}).

\begin{figure}[htbp]
    \centering
    \includegraphics[width=0.7\linewidth]{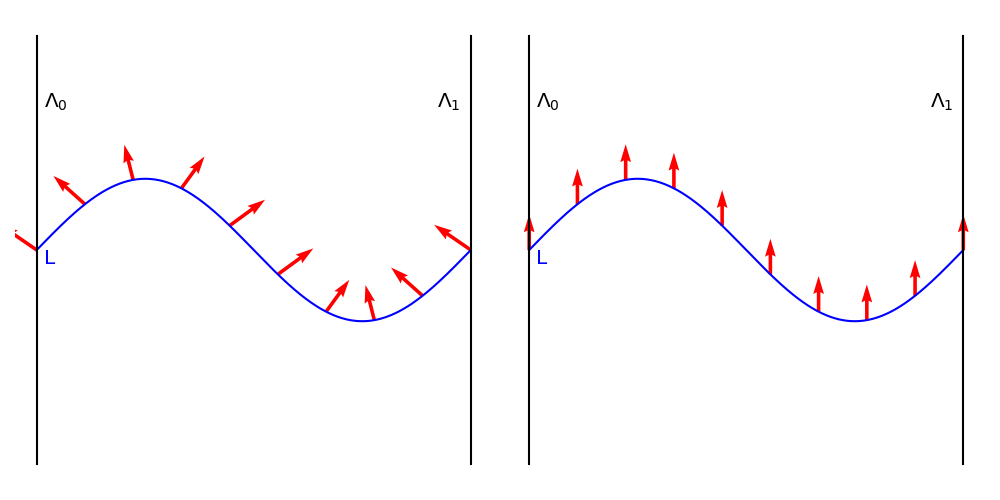}
    \caption{Consider the case $d=2$ boundary components $\Lambda_0$ and $\Lambda_1$. The flow of the normal vector field on the left does not keep the boundary of $L$ constrained along $\Lambda_0$ and $\Lambda_1$, while the vector field on the right preserves the boundary constrains $\Lambda_0, \Lambda_1$ (Definition~\ref{def:HWB-boundary-conditions-on-immersions-and-lagrangians}).}
    \label{fig:HWB-normal-vector-field-vs-boundary-cond}
\end{figure}

\begin{proposition}
         \label{prop:HWB-boundary-Lagrangian-vector-bundle-connection-2}
         Let $\Lambda_1,\dots, \Lambda_d \subset X$ be smooth and disjoint submanifolds. Then there exists a connection $\nabla^\Lambda$ on $X$ such that $\Lambda_1,\dots , \Lambda_d$ are all geodesic submanifolds for $\nabla^\Lambda$.
     \end{proposition}
     \begin{proof}
         Since $\Lambda_i$ are disjoint, can always construct a Riemannian metric such that their union is a geodesic submanifold on a tubular neighborhood of the union. Extend this Riemannian metric to all of $X$, and let $\nabla^\Lambda$ denote the corresponding Levi-Civita connection.
     \end{proof}

\begin{proposition}
\label{prop:HWB-boundary-Lagrangian-vector-bundle-connection}
         Let $\Lambda_1,\dots, \Lambda_d \subset X$ be disjoint embedded Lagrangian submanifolds of a Calabi--Yau manifold $(X,\omega, J, \Omega)$. Let $L$ be a compact manifold with smooth boundary and boundary components $C_1,\dots, C_d$, and fix a Lagrangian immersion $f:L \rightarrow X $ satisfying boundary conditions$ \Lambda_1,\dots, \Lambda_d$, i.e., $f(C_i) \subset \Lambda_i$ and $T_{f(p)}\Lambda_i \neq df(T_pL)$ for each $p\in C_i$, $1\le i\le d$.
         
         There exists a smooth function $a:L \rightarrow \RR$ such that the vector bundle $E:= (a+J)df(TL) \subset f^*TX$ satisfies the following:
             \begin{enumerate}[label={(\alph*)}]
                 \item The fiber $E_p \subset T_{f(p)} X$ is Lagrangian for each $p\in L$.
                 \item $E_p \cap df(T_pL) = \{0\}$ for all $p\in L$.
                 \item $ \dim E_p \cap T_{f(p)} \Lambda_i = 1$ for all $p\in C_i$. Furthermore $a|_{\partial L}$ is uniquely determined by this property.
                 \item Contraction by $\omega$ is an isomorphism between $E$ and $T^*L$, i.e., 
                 \[\Phi: E \xrightarrow[]{\sim} T^*L, \qquad \Phi: V \mapsto f^*i_V\omega.\]
                 \item For each boundary component $C_i \subset \partial L$, $i=1,\dots,d$,
                 \[\theta|_{T_pC_i} \equiv 0\ \iff \ \Phi|_{E}^{-1}\theta \in T_{f(p)} \Lambda_i \qquad \forall\ p\in C_i,\ \theta \in T_p^* L.\]
                 In particular if $\theta \in A^1_D(L)$ is a 1-form vanishing on the boundary $\partial L$,
                 \[(\Phi|_E^{-1} \theta)_p \in T_{f(p)} \Lambda_i \qquad\forall\ p\in C_i, \ i=1,\dots, d. \]
             \end{enumerate}
     \end{proposition}
     
     \begin{proof} If $L$ has no boundary, then $E = J df(TL)$ with $a\equiv 0$ would satisfy (a), (b) and (d). Furthermore $d=0$ and properties (c) and (e) are empty. More generally, properties (a) and (b) are true for any choice of $a$.
             \begin{enumerate}[label={(\alph*)}]
                 \item Observe that $df(TL)$ and $Jdf(TL)$ are both Lagrangian subspaces of $T_{f(p)} X$ since $f$ is a Lagrangian embedding, i.e., \[\omega(df(v_1), df(v_2)) = 0\quad \& \quad \omega(Jdf(v_1), Jdf(v_2)) = 0 \qquad \forall\ v_1,v_2 \in TL,\]
                 \begin{align*}
                     \implies \omega((a+J)dfv_1, (a+J)dfv_2) =&\ a^2\omega(dfv_1, dfv_2) + a\omega(Jdfv_1, dfv_2) \\
                     &\ + a\omega(dfv_1, Jdfv_2) + \omega(Jdfv_1, Jdfv_2)\\
                     =&\ a\omega(Jdfv_1, dfv_2) + a\omega(Jdfv_1, J^2dfv_2) \\
                     =&\ a\left(\omega(Jdfv_1, dfv_2) - \omega(Jdfv_1, dfv_2) \right)\\
                     =&\  0 \qquad \qquad \forall\ v_1,v_2 \in TL.
                 \end{align*}
                 Thus $E = (a+J)df(TL)$ is Lagrangian.
                 \item Since $df(TL)$ is Lagrangian, $df(TL) \cap Jdf(TL) = \{0\}$. It immediately follows that $E\cap df(TL) = \{0\}$ establishing (b).
                 \item By Claim~\ref{claim:HWB-boundary-choice-of-vector-bundle-a} there is a unique choice of smooth function $a|_{\partial L}$ such that $\dim E_p \cap T_{f(p)} \Lambda_i = 1$ for all $p\in C_i$, $i=1,\dots, d$. Extend $a$ smoothly to the whole of $L$ to conclude (c).
                 {
                 \begin{claim}
                 \label{claim:HWB-boundary-choice-of-vector-bundle-a}
                     Let $C_i \subset \partial L$ be a boundary component and fix $p\in C_i$. Let $v_n\in T_pL$ denote a unit normal to $T_pC_i$.
                     \begin{enumerate}[label={\roman*.}]
                         \item There is a unique number $a(p)$ such that $(a(p)+J) df(v_n) \in T_{f(p)}\Lambda_i$.
                         \item $a(p)$ is a smooth function of $p \in \partial L$.
                         \item If $E_p := (a(p) +J) df(T_pL)$, then $E_p \cap T_{f(p)} \Lambda_i = \RR (a+J) df(v_n)$.
                     \end{enumerate}
                 \end{claim}
                 \begin{proof} The strategy is to first construct such an $a(p)$, and then prove uniqueness. Smoothness of $a(p)$ will follow from the construction.
                 
                 \textbf{i.} Consider an orthogonal basis $v_1,\dots, v_{n-1}$ for $T_pC_i$ with respect to the pullback metric on $L$, and note that $v_n$ extend $v_1,\dots, v_{n-1}$ to an orthogonal basis of $T_pL$. Let \[w_i := df(v_i) \in df(T_pL).\]
                     Since $f$ is a Lagrangian immersion, $w_i$ forms an orthonormal basis for the Lagrangian plane $df(T_pL) \subset T_{f(p)}X$. Since $df(T_pL)$ is Lagrangian, $w_1,\dots, w_n, Jw_1,\dots, Jw_n$ is an orthonormal basis for $T_{f(p)} X$. By Definition~\ref{def:HWB-boundary-conditions-on-immersions-and-lagrangians}, $df(T_pC_i) \subset df(T_pL) \cap T_{f(p)} \Lambda_i$, and the linearly independent set $w_1,\dots, w_{n-1} \in T_{f(p)} \Lambda_i$ can be extended to an orthonormal basis $w_1,\dots, w_{n-1}, w$ of $T_{f(p)}\Lambda_i$. Let
                     \[w = \sum_{j=1}^n c_j w_j + d_j Jw_j \qquad c_j, d_j \in \RR.\]
                     Since $w_1,\dots, w_n, Jw_1,\dots, Jw_n$ is an orthonormal basis for $T_{f(p)}X$ and $w$ is orthogonal to $w_1,\dots, w_{n-1}$,
                     \[c_j = g(w,w_j) = 0 \qquad j=1,\dots,n-1.\]
                     Since $T_{f(p)} \Lambda_i$ is Lagrangian and $w,w_1,\dots, w_{n-1} \in T_{f(p)} \Lambda_i$, 
                     \[d_j = g(w,Jw_j) = - \omega(w, w_j) = 0 \qquad j=1,\dots, n-1.\]
                     Conclude that
                     \[w = c_nw_n + d_n Jw_n = (c_n + d_nJ)df_n(w).\]
                     Note that $d_n \neq 0$ since $T_{f(p)} \Lambda_i \neq df(T_pL)$ by Definition~\ref{def:HWB-boundary-conditions-on-immersions-and-lagrangians}. Pick $a(p):= c_n/d_n$, and observe that
                     \begin{equation}
                         \label{eq:HWB-choice-of-a-1}
                         d_n^{-1} w = (a(p)+ J) df(v_n) \in T_{f(p)} \Lambda_i. 
                     \end{equation}
                     Finally if $w' = (a'+J)df(v_n) \in T_{f(p)} \Lambda_i$, then $w' - w  = (a' - a) df(v_n) \in df(T_pL) \cap T_{f(p)} \Lambda_i$ which implies $(a'-a)v_n \in T_pC_i$. But $v_n \perp T_pC_i$ implies $a' = a$.
                     
                     \textbf{ii.} Since the chosen basis $w_1,\dots, w_{n-1}$ and $w$ varies smoothly with $p\in C_i$, the coefficients $c_n,d_n$ also vary smoothly with $p\in C_i$. So by construction $a(p)$ varies smoothly in $p \in C_i$ for each boundary component $C_i \subset \partial L$. Thus $a: \partial L \rightarrow \RR$ is smooth.

                     \textbf{iii.} Since $E = (a+J)df(TL) $, it follows from \eqref{eq:HWB-choice-of-a-1} that $(a(p)+J)df(v_n) \in E_p\cap T_{f(p)} \Lambda_i$. To see that $(a(p)+J)df(v_n)$ spans $E_p \cap T_{f(p)} \Lambda_i$, consider the bases
                    \[ (a+J)w_1,\dots, (a+J)w_n \in E_p, \quad \&\quad w_1,\dots, w_{n-1}, (a+J)w_n \in  T_{f(p)} \Lambda_i. \]
                    If $w'\in E_p \cap T_{f(p)}\Lambda_i$, then for some $c_1,\dots, c_n, d_1,\dot, d_n \in \RR$,
                    \[w' = \sum_{i=1}^n c_i (a+J)w_i \quad \&\quad w' = \sum_{i=1}^{n-1} c_i w_i + d_n(a+J)w_n. \]
                    But since $w_1,\dot, w_n, Jw_1,\dots, Jw_n$ are linearly independent,
                    
                    \[\sum_{i=1}^{n-1} (ac_i-d_i)w_i + c_iJw_i + (c_n-d_n)(a+J)w_n = 0\]
                    \[\implies \qquad c_i, d_i = 0, \quad 1\le j\le n-1, \qquad \& \qquad  c_n = d_n. \]
                    Conclude that $w' \in \RR(a+J) df(v_n) = \RR(a+j) w_n$
                 \end{proof}}

                 \item Since $E = (a+J)df(TL)$, an arbitrary element has the form $V = (a+J) df(Y)$ for some $Y \in TL$. But then
                 \[\Phi(V)(Y) = \omega(V,dfY) = \omega((a+J)dfY, dfY) = -g(dfY,dfY) \neq 0.\]
                 Conclude that $\Phi(V) \neq 0$ for any $V\in E$. But $E$ and $T^*L$ are finite-dimensional vector bundles of the same dimension. Thus $\Phi$ is an isomorphism.

                 \item Fix $p\in C_i$ and $\theta \in T_p^*L$. Let $V:= \Phi^{-1} (\theta) \in E_p$. Then
                 \[\theta(Y) = \Phi(V)(Y) = \omega(V,df(Y))  \qquad \forall\ Y\in T_pC_i, \]
                 \[\implies\qquad \theta|_{T_pC_i} \equiv 0\  \iff\ \omega(V,df(Y)) = 0 \qquad \forall\ Y\in T_pC_i. \]
                 But $f(C_i) \subset \Lambda_i$ by assumption on $f$, and so $df(T_pC_i) \subset T_{f(p)} \Lambda_i$. But by Claim~\ref{claim:HWB-boundary-choice-of-vector-bundle-a} $E_p= (a(p)+J) df(T_pL)$, and so $V = (a+J) df(Y_0)$ for some $Y_0\in T_pC_i$. Since $L$ is Lagrangian,
                 \[\begin{aligned}
                     \theta|_{T_pC_i} \equiv 0\  \iff&\ \omega(V,df(Y)) = 0 \qquad \forall\ Y\in T_pC_i \\
                     \iff&\ \omega((a+J)df(Y_0), df(Y))  = 0 \qquad \forall\ Y\in T_pC_i \\
                     \iff&\ \omega(Jdf(Y_0), df(Y))  = 0 \qquad \forall\ Y\in T_pC_i \\
                     \iff&\ g(df(Y_0), df(Y))  = 0 \qquad \forall\ Y\in T_pC_i \\
                      \iff&\ Y_0\perp T_pC_i\\
                     \iff&\ V = (a+J) df(Y_0) \in E_p \cap T_{f(p)} \Lambda_i,
                 \end{aligned}\]
                 where the last line follows from Claim~\ref{claim:HWB-boundary-choice-of-vector-bundle-a}.iii. \qedhere
             \end{enumerate}
     \end{proof}

     The bundle $E$ and connection $\nabla^\Lambda$ together define a local parametrization \eqref{eq:HWB-normal-field-to-deformation-formula-2} of a neighborhood of any given free Lagrangian immersion $f_0\in C^\infty(L, X; \Lambda_1,\dots, \Lambda_d)$ in terms of 1-forms on $L$. This parametrization can be found in Solomon--Yuval \cite[Corollary 2.10]{SolomonYuval2020-geodesics}. The parametrization is constructed in Corollary~\ref{cor:HWB-boundary-lagrangian-deformation-space}(ii). The argument to show it is indeed a parametrization is analogous to a result of Cervera--Mascar\'o--Michor \cite[Theorem 1.5]{Cervera-Mascaro-MichorMR1244452}.

     \begin{corollary}~
     \label{cor:HWB-boundary-lagrangian-deformation-space}
     \begin{enumerate}[label={(\roman*)}]
         \item Under the assumptions of Proposition~\ref{prop:HWB-boundary-Lagrangian-vector-bundle-connection}, for any 1-form vanishing on the boundary $\theta \in A^1_D(L)$, consider
     \begin{equation}
     \label{eq:HWB-normal-field-to-deformation-formula-2}
         f_{\theta}: L \rightarrow X,\qquad  f_{\theta} := \exp^\Lambda(\Phi^{-1}\theta) \circ f,
     \end{equation}
     where $f:L \rightarrow X $ is a free Lagrangian immersion satisfying boundary conditions $ \Lambda_1,\dots, \Lambda_d$ (Definition~\ref{def:HWB-boundary-conditions-on-immersions-and-lagrangians}), $\exp^\Lambda$ is the exponential map for the connection $\nabla^\Lambda$ defined in Proposition~\ref{prop:HWB-boundary-Lagrangian-vector-bundle-connection-2}, and $\Phi:E \rightarrow T^*L$ is the isomorphism from Proposition~\ref{prop:HWB-boundary-Lagrangian-vector-bundle-connection}.1(d). Then $f_{\theta}$ is a a free Lagrangian immersion with boundary conditions $ \Lambda_1,\dots, \Lambda_d$ for all $\theta$ in a sufficiently small neighborhood of zero in $A^1_D(L)$---the space of 1-forms on $L$ vanishing on the boundary $\partial L$. Furthermore $f_0 = f$, and
     \begin{equation}
         \label{eq:HWB-normal-field-to-deformation-formula-3}
         f^* \left( i_{\left. \frac{d f_{t\theta}}{dt} \right|_{t=0}} \omega \right) = \theta.
     \end{equation}
     \item Conversely if $f_1\in C^\infty(L, X; \Lambda_1,\dots, \Lambda_d)$ lies sufficiently close to $f:L \rightarrow X$ (in the $C^1$ topology), then $f_1 = f_{\theta}\circ \varphi$ for some 1-form $\theta \in A^1_D(L)$ and some diffeomorphism $\varphi \in \Diff(L)$.
     \end{enumerate}   
     \end{corollary}
     
     \begin{proof}~
     \begin{enumerate}[label={(\roman*)}]
         \item Note that since $\theta$ vanishes on the boundary, by Proposition~\ref{prop:HWB-boundary-Lagrangian-vector-bundle-connection}.1(e), $\Phi^{-1}\theta$ is tangent to $\Lambda_i$ at each boundary point $p\in C_i$. Since $\Lambda_i$ is a geodesic submanifold for the chosen connection $\nabla^\Lambda$, at any point $p\in C_i$, $\exp^\Lambda({\Phi^{-1}\theta})(p)$ lies on $\Lambda_i$. Thus
     \[f_\theta(p) = \exp^\Lambda( {\Phi^{-1} \theta}) \circ f(p) \in \Lambda_i\qquad \forall\ p\in C_i,\ i=1,\dots, d.\]
     Conclude that $f_{\theta} \in C^\infty(L, X; \Lambda_1,\dots, \Lambda_d)$ for all $\theta \in A^1_D$. Moreover \eqref{eq:HWB-normal-field-to-deformation-formula-3} follows by Proposition~\ref{prop:HWB-boundary-Lagrangian-vector-bundle-connection}.1(d) and the following calculation,
     \[ \begin{aligned}
         \left. \frac{d f_{t\theta}}{dt}\right|_{t=0} =&\ \Phi^{-1} \theta \\
         \implies \qquad f^* \left( i_{\left. \frac{d f_{t\theta}}{dt}\right|_{t=0}} \omega \right) =&\ f^* \left( i_{\Phi^{-1} \theta} \omega \right) \\
         =&\ \Phi(\Phi^{-1} \theta) = \theta.
     \end{aligned} \]

     \item The converse is a generalization of Cervera--Mascar\'o--Michor to immersions with boundary \cite[Theorem 1.5]{Cervera-Mascaro-MichorMR1244452}. A proof is provided in {\S}\ref{subsec:HWB-appendix-proof-of-converse}. \qedhere
     \end{enumerate}
     \end{proof}

  \subsection{Smooth paths and tangents in the moduli space of Lagrangians with boundary}
  \label{subsec:HWB-smooth-paths-and-tangents}

  This section recalls the notion of a smooth map of Lagrangian immersion introduced by Akveld--Salamon's \cite[p.~613]{Akveld2001}, and the tangent vector to such a smooth path \cite[Lemma~2.1]{Akveld2001}.

    \begin{definition}
    \label{def:smoothpathoflagrangians}
    Let $Y$ be a simply connected smooth manifold, and consider a map $\zcal:Y \rightarrow \lcal(X,L; \Lambda_1,\dots, \Lambda_d)$ (Definition~\ref{def:HWB-moduli-space-of-special-lagrangians}). A smooth (resp. $C^{k,\alpha}$) lifting of $\zcal$ is a smooth (resp. $C^{k,\alpha}$) function \[f:Y \times L \rightarrow X, \qquad f_y := f(y,\cdot)\]
    such that $f_y:L \rightarrow X$ is a free Lagrangian immersion representing the immersed Lagrangians submanifold $\zcal(y)$ for all $y\in Y$. The map $\zcal$ is called \textit{smooth} (resp. $C^{k,\alpha}$) if it admits a \textit{smooth (resp. $C^{k,\alpha}$) lifting}.
    \end{definition}

    Theorem~\ref{thm:HWB-tangentspacetospeciallagrangianswithboundary-2} asserts that the moduli space $\scal\lcal (X,L; \Lambda_1,\dots, \Lambda_d)$ is a smooth manifold. But Definition \ref{def:smoothpathoflagrangians} gives a different definition of smooth maps in $\scal\lcal$. It turns out that both definitions are compatible by Corollary~\ref{corollary:HWB-lifted-local-coordinate-chart} (proved later) which constructs a local coordinate chart on the smooth moduli space from Theorem~\ref{thm:HWB-tangentspacetospeciallagrangianswithboundary-2}, which is smooth in the sense of \eqref{def:smoothpathoflagrangians}.

    Akveld--Salamon identified the tangent vector to a smooth path of Lagrangian immersions $f_t:L \rightarrow X$, $t\in (-\epsilon, \epsilon)$, with the 1-form \eqref{eq:HWB-formula-for-derivative-of-path-of-Lagrangians} \cite[Lemma~2.1]{Akveld2001}. They observed that the tangent 1-form $\theta_f$ is closed and transforms well under reparametrization of the path $f_t$. The same result is reproduced in the case of Lagrangians with boundary.

  \begin{lemma}~
  \label{lemma:HWB-formula-for-derivative-of-path-of-Lagrangians}
  \begin{enumerate}[label={\arabic*.}]
      \item Consider a smooth map $f:(-\epsilon , \epsilon) \times L \rightarrow X$ such that $f_t:=f(t,\cdot) $ is a Lagrangian immersion with boundary conditions $\Lambda_1,\dots, \Lambda_d$ for each $t\in (-\epsilon, \epsilon)$ (Definition~\ref{def:HWB-boundary-conditions-on-immersions-and-lagrangians}). Then
      \begin{equation}
      \label{eq:HWB-formula-for-derivative-of-path-of-Lagrangians}
          \theta_f :=  f_0^* \left(  i_{\left. \frac{df_t}{dt}\right|_{t=0} } \omega \right)
      \end{equation}
      is a well defined closed 1-form satisfying $\theta_f|_{\partial L} \equiv 0$.
      
      \item Consider a pair of smooth maps $f, \overline{f}:(-\epsilon , \epsilon) \times L \rightarrow X$ such that $f_t := f(t,\cdot)  $ and $ \overline{f}_t:= \overline{f}(t,\cdot)$ are both Lagrangian immersions with boundary conditions $\Lambda_1,\dots, \Lambda_d$ (Definition~\ref{def:smoothpathoflagrangians}), such that $\overline{f}_t = f_t \circ \psi_t$ for some $\psi_t \in \Diff(L)$. Then
      \[ \theta_{\overline{f}} = \psi_0^* \theta_f .\]
      In particular if $\psi_0 = \id$ then $\theta_{\overline{f}} = \theta_f$. So the 1-form $\theta_f$ depends only on the path of immersed Lagrangians $[f_t]$ and the chosen immersion $f_0$ lifting the mid point $\zcal(0)$.
  \end{enumerate}
  \end{lemma}
  
   \begin{proof}~
   \begin{enumerate}[label={\arabic*.}]
       \item By Cartan's formula for the Lie derivative,
       \begin{equation}
           \label{eq:HWB-formula-for-derivative-of-path-of-Lagrangians-eq1}
           d\theta_f = df_0^* i_{\left. \frac{df_t}{dt} \right|_{t=0}} \omega = {\left. \frac{d}{dt} \right|_{t=0}} (f_t^* \omega) - f_0^*i_{\left. \frac{df_t}{dt} \right|_{t=0}} d\omega. 
       \end{equation}
       Since $\omega$ is closed, $d\omega = 0$. Since $f_t:L \rightarrow X$ is a Lagrangian immersion, $f_t^*\omega = 0$. Thus the RHS of \eqref{eq:HWB-formula-for-derivative-of-path-of-Lagrangians-eq1} vanishes proving that $\theta_f$ is closed. 
       
       Consider a boundary point $p\in \partial L$ and assume WLOG that $p\in C_i$, where $C_i$ is the boundary component of $L$ corresponding to the boundary Lagrangian $\Lambda_i \subset X$. Fix a tangent vector $\xi_p \in T_p\partial L = T_pC_i$. Then
       \begin{equation}
           \label{eq:HWB-formula-for-derivative-of-path-of-Lagrangians-eq2}
           \theta_f (\xi_p) = f_0^*\left(i_{\left. \frac{df_t}{dt} \right|_{t=0}} \omega \right) (\xi_p) = \omega \left( \left. \frac{df_t(p)}{dt} \right|_{t=0}, df_0(\xi_p) \right).
       \end{equation}
       By Definition \ref{def:HWB-boundary-conditions-on-immersions-and-lagrangians}, $f_t(C_i) \subset \Lambda_i$ for all $t\in (-\epsilon ,\epsilon)$, and in particular $f_t(p) \in \Lambda_i$ for all $t\in (-\epsilon, \epsilon)$. Then
       \begin{equation}
           \label{eq:HWB-formula-for-derivative-of-path-of-Lagrangians-eq3}
           \left. \frac{df_t(p)}{dt} \right|_{t=0} \in T_{f_0(p)}\Lambda_i \quad \&\quad df_0(\xi_p) \in T_{f_0(p)} \Lambda_i.
       \end{equation}
       But $\Lambda_i \subset X$ is Lagrangian by definition and so $\Omega|_{\Lambda_i} \equiv 0$. Combining this with \eqref{eq:HWB-formula-for-derivative-of-path-of-Lagrangians-eq2} and \eqref{eq:HWB-formula-for-derivative-of-path-of-Lagrangians-eq3}, conclude that
       \[\theta_f(\xi_p) = 0 \quad \forall\ p\in \partial L, \xi_p \in T_pL, \quad \implies \quad \theta_f|_{\partial L} \equiv 0.\]
       
       \item $\overline{f}_t = f_t\circ \psi_t$ for a smooth path $\psi_t\in \Diff(L)$. Observe that 
       \[\frac{d\overline{f}_t}{dt} = \frac{df_t}{dt}\circ \psi_t + df_t(\frac{d\psi_t}{dt}).\]
       Since $f_t$ is a Lagrangian immersion $f_t^*\omega = 0$, and so for any $\xi \in TL$,
      \[\begin{aligned}
          \overline{f}_t^* i_{\frac{d\overline{f}_t}{dt}} \omega (\xi) =&\  \overline{f}_t^* i_{\frac{df_t}{dt}\circ \psi_t} \omega (\xi) + \overline{f_t}^* i_{df_t(\frac{d\psi_t}{dt})}\omega (\xi) \\
          =&\ \overline{f}_t^* i_{\frac{df_t}{dt}\circ \psi_t} \omega (\xi) + \omega\left(df_t(\frac{d\psi_t}{dt}), d\overline{f}_t(\xi) \right) \\
          =&\ \overline{f}_t^* i_{\frac{df_t}{dt}\circ \psi_t} \omega (\xi) + \omega\left(df_t(\frac{d\psi_t}{dt}), d{f}_t \circ d\psi_t(\xi) \right) \\
          =&\  \overline{f}_t^* i_{\frac{df_t}{dt}\circ \psi_t} \omega (\xi) + (f_t^*\omega)\left(\frac{d\psi_t}{dt}, d\psi_t(\xi) \right) \\
          =&\ \overline{f}_t^* i_{\frac{df_t}{dt}\circ \psi_t} \omega (\xi),
      \end{aligned}
      \]
      \begin{equation}
          \implies \overline{f}_0^*\left(\left.i_{\frac{d\overline{f}_t}{dt}} \omega\right|_{t=0} \right) = \psi_0^* f_0^*\left(\left.i_{\frac{df_t}{dt}} \omega\right|_{t=0} \right). \tag*{\qedhere}
      \end{equation}
   \end{enumerate}
   \end{proof}

\subsection{Hodge decomposition for manifolds with boundary}
\label{subsec:hodgedecomposition}

Section \ref{subsec:hodgedecomposition} recalls Hodge theory for manifolds with boundary \cite[Chapter 5, Proposition 9.8]{Taylor2011}. The reader is referred to Petersen \cite[Chapter 7]{Petersen2006} 
and Taylor \cite[Chapter 5]{Taylor2011} more details and a proof of the Hodge decomposition theorem.
%

\begin{definition}
\label{def:ofhodgestaronmanifolds}
Let $(M,g)$ be an n-dimensional compact oriented Riemannian manifold with boundary. For every $0\le k\le n$, the \textit{Hodge star operator} is the unique linear map of vector bundles $\star:\Lambda^kT^*M \rightarrow \Lambda^{n-k}T^*M$ satisfying
\begin{equation}
    \omega_1\wedge \star \omega_2 = g(\omega_1, \omega_2) d\vol \qquad \forall\ \omega_1,\omega_2 \in A^k(M).
    \label{eq:definitionofhodgestaronmanifolds}
\end{equation}
where $A^k(M)$ is the vector space of $k$-forms on $M$. Existence and uniqueness follows from non-degeneracy of $g$.
\end{definition}



The relative De Rham cohomology and De Rham cohomology classes are represented respectively by harmonic forms with Dirichlet and Neumann boundary conditions. Dirichlet and Neumann boundary conditions on differential forms are recalled in Definition~\ref{def:boundarydecompositionofformsontheboundary} \cite[p.~422]{Taylor2011}.

\begin{definition}
\label{def:boundarydecompositionofformsontheboundary}
Let $(M,g)$ be a compact Riemannian manifold with boundary, and $p\in \partial M$ a boundary point.
\begin{enumerate}[label={(\arabic*)}]
    \item For $0\le k\le n$, $\Lambda^kT^*_PM$ admits an orthogonal decomposition
    \[\Lambda^kT^*_pM = \Lambda^kT^*_p\partial M \oplus \left( \Lambda^kT^*_p\partial M\right)^{\perp}. \]
    For any $\omega \in \Lambda^k T_p^*M$, the orthogonal decomposition of forms is denoted by
    \[\omega = \omega_{tan} + \omega_{norm},\qquad \omega_{tan} \in \Lambda^kT^*_p\partial M,\ \omega_{norm} \in \left( \Lambda^kT^*_p\partial M\right)^{\perp}.\]
    $\omega_{tan}$ is called the tangential component of $\omega$ and $ \omega_{norm}$ is called the the {normal component} of $\omega$.
    \item The \textit{ Dirichlet boundary condition} on a $k$-forms $\omega \in A^k(M)$ is $(\omega_p)_{tan} \equiv 0$ for all $p\in \partial M$ and the Neumann boundary condition on a $k$-form $\omega \in A^k(M)$ is $(\omega_p)_{norm} \equiv 0$ for all $p\in \partial M$. The space of Dirichlet and Neumann $k$-forms are respectively denoted
\begin{align}
    A^k_D(M) &:= \{\omega \in A^k(M): \omega_{tan} = 0\}, \label{eq:dirichletboundary}\\
    A^k_N(M) &:= \{\omega \in A^k(M): \omega_{norm} = 0\} \label{eq:neumannboundary}.
\end{align}
    
\end{enumerate}
\end{definition}

By Lemma~\ref{lemma:hodgestaroperatorattheboundary}(i) asserts that the Dirichlet condition on differential forms is independent of the chosen Riemannian metric. However the Neumann condition depends on the choice of Riemannian metric. Lemma~\ref{lemma:hodgestaroperatorattheboundary}(ii) proves that the Hodge star operator carries {Dirichlet forms} to {Neumann forms}, and vice-versa.

\begin{lemma}
\label{lemma:hodgestaroperatorattheboundary}
Let $(M,g)$ be a compact Riemannian manifold with boundary, and let $I:\partial M \rightarrow M$ be the boundary inclusion map. For a $k$-form $\omega \in A^k(M)$,
\begin{enumerate}[label={(\roman*)}]
    \item $I^*\omega = I^*\omega_{tan}$. In particular $I^*\omega = 0$ if and only if $\omega_{tan} = 0$,
    \item $\omega_{tan}\equiv 0$ if and only if $(\star \omega)_{norm} \equiv 0$.
\end{enumerate}
\end{lemma}
\begin{proof}
Lemma~\ref{lemma:hodgestaroperatorattheboundary} is a local statement and it is sufficient to prove it at each point of $\partial M$. Fix $p\in \partial M$ and $U \subset M$ an open neighborhood of $p$ endowed with a boundary coordinate chart $\varphi:U \rightarrow \RR^n$ based at $p$, i.e., for $x_1,\dots, x_n$ coordinates on $\RR^n$,
\[\varphi(U) = \{x_n \ge 0\}\cap \varphi(U), \quad \varphi(p) = 0, \quad \& \quad \varphi(\partial M \cap U) = \{x_n = 0\} \cap \varphi(U).\]
The chart $\varphi$ can be composed with a linear map if necessary to ensure $dx_1,\dots, dx_n$ is an orthonormal basis of $T^*_pM$. Then $\{dx_{i_1}\wedge \dots \wedge dx_{i_k}: 1\le i_1< \dots < i_k\le n \}$ is an orthonormal basis for $\Lambda^k T^*_pM$. The Hodge star operator acts on such a basis vector to produce another basis vector with the complementary set of indices up to a sign. In particular if $i_1,\dots, i_k \le n-1$ and $\{j_1,\dots, j_{n-k-1}\} = \{1,\dots, n-1\} \setminus \{i_1,\dots, i_{k}\}$, then
\begin{gather}
    \label{eq:hodgestaronboundaryONB}
    \star(dx_{i_1}\wedge \dots \wedge dx_{i_k}) = \pm dx_{j_1} \wedge \dots \wedge dx_{j_{n-k-1}}\wedge dx_n
\end{gather}
Note that the Hodge star swaps basis vectors involving $dx_n$ those not involving it. At $p\in \partial M$ let
\begin{align*}
    \omega\ = \quad &\ \sum_{1\le i_1< \dots< i_k\le n-1} a_{i_1,\dots, i_k}dx_{i_1}\wedge \dots \wedge dx_{i_k} \\
    + &\ \sum_{1\le j_1< \dots< j_{k-1}\le n-1} b_{j_1,\dots, j_{k-1}}dx_{j_1}\wedge \dots \wedge dx_{j_{k-1}}\wedge dx_n, 
\end{align*}
for some $a_{i_1,\dots,i_k}, b_{j_1,\dots, j_{k-1}}\in \RR$. Since $x_n$ restricts to zero on the boundary $\partial M$,
\begin{equation}
    (I^*\omega)_p =  \sum_{1\le i_1< \dots< i_k\le n-1}a_{i_1,\dots, i_k}dx_{i_1}\wedge \dots \wedge dx_{i_k}. \label{eq:boundaryONBpullback}
\end{equation}
But since $dx_1,\dots, dx_n$ is an orthonormal basis for $T_p^*M$ and $dx_1,\dots, dx_{n-1}$ is an orthonormal basis for $ T_p^*\partial M$,
\begin{align}
    (\omega_{tan})_p =&\ \sum_{1\le i_1< \dots< i_k\le n-1}a_{i_1,\dots, i_k}dx_{i_1}\wedge \dots \wedge dx_{i_k}, \label{eq:boundaryONBtangential} \\
    (\omega_{norm})_p =&\ \sum_{1\le i_1< \dots< i_{k-1}\le n-1}b_{i_1,\dots, i_k}dx_{i_1}\wedge \dots \wedge dx_{i_{k-1}}\wedge dx_n, \label{eq:boundaryONBnormal}.
\end{align}
Lemma~\ref{lemma:hodgestaroperatorattheboundary}(i) follows from \eqref{eq:boundaryONBpullback} and \eqref{eq:boundaryONBtangential}. Lemma~\ref{lemma:hodgestaroperatorattheboundary}(ii) follows from \eqref{eq:hodgestaronboundaryONB}, \eqref{eq:boundaryONBtangential}, and \eqref{eq:boundaryONBnormal}.
\end{proof}

Recall the {co-exterior derivative operator} $d^*: A^k(M) \rightarrow A^{k-1}(M)$ defined in terms of the Hodge star operator $\star$ and the exterior derivative by
\begin{equation}
    \label{eq:dstaroperatordefinition}
    d^* := (-1)^{n(k+1)+1} \star d \star.
\end{equation}
The sign in \eqref{eq:dstaroperatordefinition} is chosen to make $d^*$ the formal adjoint of $d$ with respect to the $L^2$ inner product on $k$-forms vanishing at the boundary. Definition \ref{def:HWB-vector-spaces-of-forms-and-boundary-conditions} introduces notation for closed, exact, Dirichlet and Neumann forms which is used to state the Hodge decomposition Theorem~\ref{thm:hodgedecompositiontheoremformanifoldswithboundary} and Corollary \ref{cor:harmonicfieldsvanishingontheboundaryandcohomology}.

\begin{definition}
\label{def:HWB-vector-spaces-of-forms-and-boundary-conditions}
Let $(M,g)$ be a compact manifold with boundary. The \textit{Laplace-Beltrami operator for $k$-forms on $M$} is by definition
\begin{equation}
    \Delta : A^k(M) \rightarrow A^k(M),\qquad \Delta := d d^* + d^*d\qquad k\ge 0.
    \label{eq:definitionofharmonicformonmanifoldwithboundary}
\end{equation}
A $k$-form $\omega \in A^k(M)$ is called harmonic if $\Delta\, \omega = 0$. Denote the space of closed $k$-forms and exact $k$-forms respectively by
\begin{equation}
    \label{eq:closedandexactkforms}
    Z^k(M) = \ker d\subset A^k(M) \qquad \& \qquad E^k(M) := d(A^{k-1}(M)).
\end{equation}
 Similarly denote the space of co-closed and co-exact $k$-forms by
 \begin{equation}
    \label{eq:coclosedandcoexactkforms}
    cZ^k(M) = \ker d^* \subset A^k(M) \qquad \& \qquad cE^k(M) := d^*(A^{k+1}(M))
\end{equation}
Recall \eqref{eq:dirichletboundary} and \eqref{eq:neumannboundary} and let
\[\begin{aligned}
    Z^k_D(M) := Z^k(M) \cap A^k_D(M) \quad \&&\ \quad Z^k_N(M) := Z^k(M) \cap A^k_N(M) \\
    cZ^k_D(M) := cZ^k(M) \cap A^k_D(M) \quad \&&\ \quad cZ^k_N(M) := cZ^k(M) \cap A^k_N(M) 
\end{aligned}\]
To state the Hodge decomposition theorem for manifolds with boundary (Theorem~\ref{thm:hodgedecompositiontheoremformanifoldswithboundary}) also consider the spaces of exact Dirichlet and Neumann $k$-forms, 
\begin{align}
    E_N^k(M) := d(A_N^{k-1}),\quad &  E_D^k(M) := d(A_D^{k-1}(M)), \label{eq:exactwithboundaryconditions}\\ 
    cE_N^k(M) := d^*(A_N^{k+1}),\quad & cE_D^k(M) := d^*(A_D^{k+1}(M)) \label{eq:coexactwithboundaryconditions}.
\end{align}
Note that the boundary condition is applied before taking exterior derivatives.
\end{definition}

Closed $(d\omega = 0)$ and co-closed $(d^*\omega = 0)$ differential forms are always harmonic and on a closed manifolds, these are the only harmonic forms. However a manifold with boundary might admit harmonic forms that fail to be either closed or or co-closed. The closed and co-closed harmonic forms $Z^k(M)\cap cZ^k(M)$ form a distinguished subspace of all harmonic $k$-forms.

\begin{definition}
\label{def:harmonicfields}
A harmonic $k$-form that is both closed and co-closed is called a \textit{harmonic field}.
\end{definition}

Hodge decomposition for manifolds with boundary is stated using notation from \eqref{eq:dirichletboundary}, \eqref{eq:neumannboundary}, \eqref{eq:closedandexactkforms}, \eqref{eq:coclosedandcoexactkforms}, \eqref{eq:exactwithboundaryconditions}, and \eqref{eq:coexactwithboundaryconditions}. For a detailed proof, refer to Taylor \cite[Chapter 5, Proposition 9.8]{Taylor2011}.

\begin{theorem}[Hodge decomposition for manifolds with boundary]
\label{thm:hodgedecompositiontheoremformanifoldswithboundary}
Let $(M,g)$ be  compact, connected, oriented, smooth Riemannian manifold with boundary. Then for every $k\ge 0$ the following is an orthogonal direct sum decomposition
\[A^k = cE^k_N\oplus (Z^k\cap cZ^k) \oplus E^k_D.\]
The space of harmonic fields has a further decomposition
\[Z^k\cap cZ^k = (Z^k\cap cZ^k \cap A^k_N) \oplus (E^k\cap cZ^k) = (Z^k\cap cE^k) \oplus (Z^k\cap cZ^k\cap A^k_D).\]
The manifold $(M,g)$ is omitted in the previous equations for visual clarity.
\end{theorem}

As a corollary of Theorem~\ref{thm:hodgedecompositiontheoremformanifoldswithboundary} each relative De Rham cohomology class admits a unique representative harmonic field with Dirichlet boundary conditions, and each De Rham cohomology class admits a unique representative harmonic field with Neumann boundary conditions.


\begin{corollary}
\label{cor:harmonicfieldsvanishingontheboundaryandcohomology}
Let $(M,g)$ be a smooth, connected, compact Riemannian manifold with boundary. Then,
\begin{enumerate}[label={(\roman*)}]
    \item $Z^k\cap cZ^k\cap A_D^k(M) \cong H^k(M,\partial M; \RR)$ via the map taking $\alpha \in Z^k\cap cZ^k \cap A^k_D$ to its relative cohomology class $[\alpha] \in H^k(M,\partial M; \RR)$.
    \item $Z^k\cap cZ^k\cap A_N^k(M) \cong H^k(M; \RR)$ via the map taking a closed $k$-form $\alpha$ to its de Rahm cohomology class $[\alpha] \in H^k(M;\RR)$.
\end{enumerate}
\end{corollary}
\begin{proof}~
\begin{itemize}
    \item[(i)] The relative cohomology $H^k(M,\partial M; \RR)$ is isomorphic to the homology of the chain complex $(A^k_D(M), d)$ of differential forms vanishing on the boundary along with the exterior derivative \cite[Chapter 5, Proposition~9.9]{Taylor2011}. With notation as in Definition~\ref{def:HWB-vector-spaces-of-forms-and-boundary-conditions}, $H^k(M,\partial M;\RR)$ is the quotient $Z^k\cap A^k_D/ E^k_D$. By Theorem~\ref{thm:hodgedecompositiontheoremformanifoldswithboundary},
    \begin{equation}
        \label{eq:hodgedecompositioncorollaryeq1}
        A^k = cE^k_N \oplus (Z^k\cap cE^k) \oplus (Z^k\cap cZ^k \cap A^k_D)\oplus E^k_D(M).
    \end{equation}
\begin{claim}
\label{hodgedecompcorrollaryclaim1}
$Z^k\cap cE^k_N =0$.
\end{claim}
\begin{proof}
Let $\alpha \in Z^k\cap cE^k_N$. Let $\alpha = d^*\beta$ for some $\beta \in A^{k+1}_N$. By Lemma~\ref{lemma:hodgestaroperatorattheboundary}, $(\star \beta)_{tan} \equiv 0$ and so $(\star \beta)|_{\partial M} \equiv 0$.
\begin{align*}
    \int_M \langle \alpha,\alpha \rangle d\vol =&\ \int_M \alpha\wedge  \star (d^*\beta) = (-1)^{n(k+2)+1+k(n-k)} \int_M \alpha \wedge d( \star \beta) \\
    =&\ (-1)^{n(k+2)+1+k(n-k)} \left(\int_M - (d\alpha)\wedge \star \beta + \int_M d(\alpha \wedge \star \beta)\right)\\
    =&\ 0+ (-1)^{2n+1-k^2} \int_{\partial M}\alpha \wedge \star \beta = 0,
\end{align*}
because $d\alpha = 0 $, $(\star \beta)|_{\partial M} \equiv 0$ vanishes on the boundary. Thus $\alpha = 0$.
\end{proof}
Combining \eqref{eq:hodgedecompositioncorollaryeq1} and Claim \ref{hodgedecompcorrollaryclaim1}, 
\begin{equation}
    \label{eq:hodgedecompositioncorollaryeq12}
    Z^k\cap A^k_D = (Z^k\cap A^k_D\cap cE^k) \oplus (Z^k\cap cZ^k\cap A^k_D) \oplus E^k_D.
\end{equation}
\begin{claim}
\label{hodgedecompcorrollaryclaim2}
$Z^k\cap cE^k \cap A^k_D =0$.
\end{claim}
\begin{proof}
 Fix $\alpha \in Z^k\cap cE^k\cap A^k_D$ and suppose $\alpha = d^*\beta$ for some $\beta \in A^{k+1}(M)$. By Stokes' theorem, since $\alpha \in A^k_D$
\begin{align*}
    \int_M\langle \alpha,\alpha \rangle d\vol =&\ \int_M \alpha\wedge \star (d^*\beta) = (-1)^{n(k+2)+1+k(n-k)} \int_M \alpha \wedge d(\star \beta) \\
    =&\ (-1)^{n(k+2)+1+k(n-k)} \left\{\int_M - (d\alpha)\wedge *\beta + \int_M d(\alpha \wedge \star \beta)\right\}\\
    =&\ 0+ (-1)^{2n+1-k^2} \int_{\partial M}\alpha \wedge \star \beta = 0
\end{align*}~
\begin{equation}
    \implies \alpha = 0. \tag*{\qedhere}
\end{equation}
\end{proof}
As a consequence of Claim \ref{hodgedecompcorrollaryclaim2} and \eqref{eq:hodgedecompositioncorollaryeq12},
\[Z^k\cap A^k_D = Z^k\cap cZ^k\cap A^k_D \oplus E^k_D\]
\[\implies H^k(M,\partial M;\RR) \cong \frac{Z^k\cap A^k_D}{E^k_D} \cong Z^k\cap cZ^k\cap A^k_D.\]

\item[(ii)] By De Rham's theorem $H^k(M;\RR) \cong Z^k /E^k$. Similar to part (i), by Theorem~\ref{thm:hodgedecompositiontheoremformanifoldswithboundary} and Claim \ref{hodgedecompcorrollaryclaim1}
\begin{equation}
    \label{eq:hodgedecompositioncorollaryeq13}
    Z^k = (Z^k\cap cZ^k\cap A^k_N) \oplus (E^k\cap cZ^k) \oplus E^k_D.
\end{equation}
But $E^k\cap cZ^k \oplus E^k_D \subset E^k$, so taking a quotient of \eqref{eq:hodgedecompositioncorollaryeq13},
\begin{equation}
    H^k(M) \cong \frac{Z^k}{E^k} \cong Z^k(M)\cap cZ^k(M) \cap A^k_N(M). \tag*{\qedhere}
\end{equation}
\end{itemize}
\end{proof}

\section{Moduli space of special Lagrangians}
\label{sec:HWB-moduli-space-of-special-lagrangians}

This section is concerned with Problem~\ref{problem:HWB-3.1} and contains proofs of Proposition~\ref{prop:HWB-tangentspacetospeciallagrangianswithboundary-1} and Theorem~\ref{thm:HWB-tangentspacetospeciallagrangianswithboundary-2}.

\subsection{Hodge dual of tangent 1-form}
\label{subsec:HWB-hodge-dual-of-tangent-1-form}

The two claims of Proposition~\ref{prop:HWB-tangentspacetospeciallagrangianswithboundary-1} are proved separately as Lemma~\ref{lemma:HWB-phit-is-closed} and Lemma~\ref{lemma:HWB-phitishodgestarofthetat}. The proof of Lemma~\ref{lemma:HWB-phit-is-closed} is analogous to the proof of Lemma~\ref{lemma:HWB-formula-for-derivative-of-path-of-Lagrangians}.

\begin{lemma}
\label{lemma:HWB-phit-is-closed}
Under the assumptions of Proposition~\ref{prop:HWB-tangentspacetospeciallagrangianswithboundary-1}, the $(n-1)$-form $\phi_t$ defined in \eqref{eq:HWB-thm-1-theta-phi-notation} is closed.
\end{lemma}
\begin{proof}
Recall $f_t:= f(t,\cdot):L \rightarrow X$ is a smooth path of special Lagrangian immersions. By Cartan's formula for the Lie derivative,
\[\frac{d}{dt} f_t^*\Im(\Omega) = f_t^*\left(i_{\frac{df_t}{dt}} d \ \Im(\Omega) + d i_{\frac{df_t}{dt}}\ \Im(\Omega) \right).\]
\begin{align*}
    \implies \qquad d\phi_t = d\, f_t^*\left(i_{\frac{df_t}{dt}} \Im(\Omega) \right) =&\ f_t^*\left(d\, i_{\frac{df_t}{dt}} \Im(\Omega) \right) \\
    =&\ \frac{d}{dt} f_t^*\Im(\Omega) - f_t^*\left(i_{\frac{df_t}{dt}} d \ \Im(\Omega)  \right)
\end{align*}
But $d \Im(\Omega)=0$ since $\Omega$ is holomorphic and hence closed, and $f_t^*\Im(\Omega) = 0$ for all $t\in[0,1]$ since $f_t$ is a special Lagrangian immersion. Thus $d\phi_t = 0$.
\end{proof}

 Recall the forms $\theta_t$ and $\phi_t$ as in \eqref{eq:HWB-thm-1-theta-phi-notation}. Lemma~\ref{lemma:HWB-phitishodgestarofthetat}  asserts that For any fixed $t_0 \in [0,1]$, $\theta_{t_0}$ and $\phi_{t_0}$ are Hodge dual to each other whenever $f_{t_0}$ is a special Lagrangian immersion. This is a more general statement than Proposition~\ref{prop:HWB-tangentspacetospeciallagrangianswithboundary-1}.2, and the more general statement is used in the proof of Theorem~\ref{thm:HWB-tangentspacetospeciallagrangianswithboundary-2}.

 \begin{lemma}
    \label{lemma:HWB-phitishodgestarofthetat} 
    Let $(X,\omega, J, \Omega)$ be a Calabi--Yau manifold, and $L$ a compact manifold with boundary. Fix pairwise disjoint embedded Lagrangians $\Lambda_1,\dots, \Lambda_d \subset X$. Consider a smooth function $f: (-\epsilon, \epsilon) \times L \rightarrow X$ such that $f_t:=f(t,\cdot) \in C^\infty(L, X; \Lambda_1,\dots, \Lambda_d)$ for all $t\in (-\epsilon,\epsilon)$ (Definition~\ref{def:HWB-boundary-conditions-on-immersions-and-lagrangians}). Assume that $f_0$ is a special Lagrangian immersion, i.e., $f_0^* \Im(\Omega) \equiv 0$. Consider the following pair of differential forms on $L$,
    \begin{equation*}
        \begin{aligned}
        \theta_0:=&\ f_0^* i_{ \left. \frac{df_t}{dt} \right|_{t=0} } \omega \in A^1(L) \\
        \phi_0 :=&\ f_0^* i_{ \left. \frac{df_t}{dt} \right|_{t=0} } \Im(\Omega) \in A^{n-1}(L).
    \end{aligned}
    \end{equation*}
    Then $\phi_0 = \star_0 \theta_0$ where $\star_0$ is the Hodge star for the pullback metric $f_0^*g$. 
 \end{lemma}

\begin{proof}
Let $\xi_0$ denote the vector field dual to the 1-form $\theta_0$ with respect to $f_0^* g$, i.e.,
\[\theta_0(Y) = f_0^*g(\xi_0, Y) \qquad \forall\ Y\in TL.\]
The goal is to compute $\star_0 \theta_0$ and show it equals $\phi_0$. The argument is split into Claims~\ref{claim:HWB-phi-hodgestar-of-theta-claim1}--\ref{claim:HWB-phi-hodgestar-of-theta-claim4}. These are first used to prove Lemma~\ref{lemma:HWB-phitishodgestarofthetat}, and four claims are proved subsequently.

\begin{claim}
\label{claim:HWB-phi-hodgestar-of-theta-claim1}
    $\displaystyle{i_{JY}\Im(\Omega) = i_Y \Re(\Omega)}$ for all $Y\in TX$.
\end{claim}

\begin{claim}
    \label{claim:HWB-phi-hodgestar-of-theta-claim2}
    $\displaystyle{\star_0 \theta_0 = i_{\xi_0} f_0^* \Re(\Omega).}$
\end{claim}

\begin{claim}
\label{claim:HWB-phi-hodgestar-of-theta-claim3}
    $\displaystyle{Jdf_0(\xi_0) - \left. \frac{df_t}{dt}\right|_{t=0} = df_0(\chi_0)}$ for some vector field $\chi_0\in TL$.
\end{claim}

\begin{claim}
\label{claim:HWB-phi-hodgestar-of-theta-claim4}
$\displaystyle{f_0^*\left( i_{ \left. \frac{df_t}{dt} \right|_{t=0}} \Im(\Omega) \right) = f_0^*\left( i_{Jdf_0(\xi_t)} \Im(\Omega) \right)}$
\end{claim}

By Claims~\ref{claim:HWB-phi-hodgestar-of-theta-claim1} and \ref{claim:HWB-phi-hodgestar-of-theta-claim2},
\begin{align*}
    \star_0 \theta_0 =&\ i_{\xi_0} f_0^* \Re(\Omega) \\
    =&\ f_0^*\left( i_{df_0(\xi_0)} \Re(\Omega) \right) \\
    =&\ f_0^* \left( i_{Jdf_0(\xi_0)}\Im(\Omega) \right).
\end{align*}
Note that the previous step uses the assumption that $f_0$ is special Lagrangian to apply Claim~\ref{claim:HWB-phi-hodgestar-of-theta-claim2}. Finally by Claim~\ref{claim:HWB-phi-hodgestar-of-theta-claim4} and the definition of $\phi_0$,
\begin{align*}
    \star \theta_0  =&\ f_0^* \left( i_{Jdf_0(\xi_0)}\Im(\Omega) \right) =  \\
    =&\ f_0^*\left( i_{ \left. \frac{df_t}{dt} \right|_{t=0}} \Im(\Omega) \right) = \phi_0.
\end{align*}
This completes the proof of Lemma~\ref{lemma:HWB-phitishodgestarofthetat}.
\end{proof}

\begin{proof}[Proof of Claim~\ref{claim:HWB-phi-hodgestar-of-theta-claim1}]
    Claim~\ref{claim:HWB-phi-hodgestar-of-theta-claim1} is a general statement about Calabi--Yau manifolds. Since $\Omega$ is an $(n,0)$-form, $i_{JY} \Omega = \sqrt{-1} i_Y\Omega$ for any $Y\in TX$.
    \begin{equation}
        i_{JY}\Im(\Omega) = \Im\left(i_{JY}\Omega \right) = \Im\left(\sqrt{-1} i_Y\Omega \right) = i_Y \Re(\Omega). \tag*{\qedhere}
    \end{equation}
\end{proof}

\begin{proof}[Proof of Claim~\ref{claim:HWB-phi-hodgestar-of-theta-claim2}]
Claim~\ref{claim:HWB-phi-hodgestar-of-theta-claim2} uses the assumption that $f_0$ is a special Lagrangian immersion by invoking Lemma~\ref{lemma:HWB-harveylawson-calibration-lemma}. By Definition \ref{def:ofhodgestaronmanifolds}
\[ \beta \wedge \star_0\theta_0 = \langle \beta, \theta_0 \rangle_{f_0^*g} \vol_{f_0^*g} = \left( i_{\xi_0}\beta \right)\vol_{f_0^*g} \qquad \forall\ \beta \in A^1(L). \]
By the product rule for interior products
\[0 = i_{Y}(\beta\wedge \vol_{f_0^*g}) = (i_{Y}\beta) \vol_{f_0^*g} - \beta \wedge i_{Y}\vol_{f_t^*g} \qquad \forall \ Y\in TL\]
\begin{align*}
    \implies \beta \wedge \star_0\theta_0  =&\ \left(i_{\xi_0}\beta \right) \vol_{f_0^*g} = \beta \wedge i_{\xi_0} \vol_{f_0^*g}, \\
     =&\ \beta\wedge i_{\xi_0} f_0^* \Re(\Omega), \qquad \forall\ \beta\in A^1(L).
\end{align*}
where the last line follows from Lemma~\ref{lemma:HWB-harveylawson-calibration-lemma} applied to the special Lagrangian immersion $f_0$. Since $\beta \in A^1(L)$ is arbitrary Claim~\ref{claim:HWB-phi-hodgestar-of-theta-claim2} follows.
\end{proof}

\begin{proof}[Proof of Claim~\ref{claim:HWB-phi-hodgestar-of-theta-claim3}]
    Claim~\ref{claim:HWB-phi-hodgestar-of-theta-claim3} is used to prove Claim~\ref{claim:HWB-phi-hodgestar-of-theta-claim4} and is property of the form $\theta_0$. By the definition of $\xi_0$ and $\theta_0$ \eqref{eq:HWB-formula-for-derivative-of-path-of-Lagrangians}
    \[\begin{aligned}
    g(df_0(\xi_0), df_0(Y)) = \theta_0(Y) =&\ f_0^* \left( i_{\frac{df_t}{dt}|_{t=0}} \omega \right)(Y) \\
    =&\ \omega\left(\left. \frac{df_t}{dt} \right|_{t=0}, df_t(Y) \right) \\
    =&\ g\left(-J\, \left. \frac{df_t}{dt} \right|_{t=0}, df_0(Y)\right) \qquad \forall\ Y\in TL.
    \end{aligned}
    \]
    The vector field $df_0(\xi_0) + J \left. \frac{df_t}{dt} \right|_{t=0}$ is orthogonal to $df_0(TL)$. Since $df_0(T_pL)$ is a Lagrangian subspace of $T_{f_0(p)}X$ for each $p\in L$, $Jdf_0(\xi_0) - \left. \frac{df_t}{dt} \right|_{t=0} \in df_0(TL)$. Thus there exists a vector field $\chi_0\in TL$ such that $Jdf_t(\xi_0) - \left. \frac{df_t}{dt} \right|_{t=0} = df_t(\chi_0)$.
\end{proof}

\begin{proof}[Proof of Claim~\ref{claim:HWB-phi-hodgestar-of-theta-claim4}]
    Recall that $df_0(T_pL)$ is a special Lagrangian plane and $f_0^*\Im(\Omega) = 0$ since $f_0$ is a special Lagrangian immersion. 
\[\Im(\Omega) (df_0(Y_1),\dots, df_0(Y_n)) = 0 \qquad \forall\ Y_1,\dots, Y_n \in TL. \]
Combining this observation for $Y_1 = \chi_0$, along with Claim~\ref{claim:HWB-phi-hodgestar-of-theta-claim3},
\begin{align*}
    \Im(\Omega) \left(\left. \frac{df_t}{dt} \right|_{t=0} - Jdf_0(\xi_0), df_0(Y_2), \dots, df_0(Y_{n})\right) =&\ \Im(\Omega) \left(df_0(\chi_0), df_0(Y_2), \dots, df_0(Y_{n} ) \right) \\
    =&\ 0 \qquad \forall\ Y_2,\dots, Y_n \in TL \\
    \implies\qquad  f_0^*\left( i_{ \left. \frac{df_t}{dt} \right|_{t=0} - Jdf_0(\xi_0)} \Im(\Omega) \right) =&\ 0
\end{align*}
\begin{equation}
        \implies \qquad f_0^*\left( i_{ \left. \frac{df_t}{dt} \right|_{t=0}} \Im(\Omega) \right) = f_0^*\left( i_{Jdf_0(\xi_t)} \Im(\Omega) \right).\tag*{\qedhere}
    \end{equation}
\end{proof}

\subsection{Proof of Proposition~\ref{prop:HWB-tangentspacetospeciallagrangianswithboundary-1}}
\label{subsec:HWB-proof-of-thm-1-1}

\begin{proof}[\text{Proof of Proposition~\ref{prop:HWB-tangentspacetospeciallagrangianswithboundary-1}}]
    By Lemma~\ref{lemma:HWB-formula-for-derivative-of-path-of-Lagrangians}, $\theta_t$ is a closed 1-form and pulls back to zero on the boundary $\partial L$. By Lemma~\ref{lemma:hodgestaroperatorattheboundary} $\theta_t$ satisfies Dirichlet boundary conditions (Definition~\ref{def:boundarydecompositionofformsontheboundary}(3)). By Lemma~\ref{lemma:HWB-phit-is-closed}, $\phi_t$ is a closed $(n-1)$-form. Finally, $\phi_{t_0} = \star_{t_0} \theta_{t_0}$ follows by applying Lemma~\ref{lemma:HWB-phitishodgestarofthetat} along the path $f_t$ for each $t_0\in[0,1]$.
\end{proof}

\subsection{Special Lagrangian operator}
\label{subsec:HWb-special-Lagrangian-operator}
     
     This section constructs the special Lagrangian operator \eqref{eq:HWB-linearization-of-SLAG-2} and proves that it is elliptic (Proposition~\ref{prop:HWB-ellipticity-of-SLAG-operator}).

     Fix a free special Lagrangian immersion $f_0:L \rightarrow X$ representing an immersed special Lagrangian $\zcal_0 \in \scal\lcal(X,L:\Lambda_1,\dots,  \Lambda_d)$. Recall the connection $\nabla^\Lambda$ from Propositions~\ref{prop:HWB-boundary-Lagrangian-vector-bundle-connection-2} and the vector bundle $E$ from Propositions~\ref{prop:HWB-boundary-Lagrangian-vector-bundle-connection}. By Corollary \ref{cor:HWB-boundary-lagrangian-deformation-space}, for all $\theta \in A^1_D(L)$, $f_\theta:L \rightarrow X$---defined by \eqref{eq:HWB-normal-field-to-deformation-formula-2}---is a free immersion satisfying boundary conditions $\Lambda_1,\dots, \Lambda_d$. \eqref{eq:HWB-normal-field-to-deformation-formula-2} is recalled again for the reader's convenience:
     \[f_{\theta}: L \rightarrow X,\qquad  f_{\theta} := \exp^\Lambda(\Phi^{-1}\theta ) \circ f,\]
     where $\exp^\Lambda$ is the exponential map for $\nabla^\Lambda$ from Proposition~\ref{prop:HWB-boundary-Lagrangian-vector-bundle-connection-2}, and $\Phi$ is defined by Proposition~\ref{prop:HWB-boundary-Lagrangian-vector-bundle-connection}(d). $A^1_D(L)$ is the natural domain for the \textit{special Lagrangian operator} operator
     \begin{equation}
         \label{eq:HWB-linearization-of-SLAG-2}
         F:A^1_D(L) \rightarrow A^2(L)\oplus A^{n}(L), \qquad F(\theta) := f_{\theta}^*\omega \oplus f_{\theta}^* \Im(\Omega).
     \end{equation}
     For sufficiently small $\theta \in A^1_D(L)$, $f_\theta$ is a special Lagrangian immersion with boundary conditions $\Lambda_1,\dots, \Lambda_d$ if and only if $F(\theta) = 0$ by Definition~\ref{def:immersedlagrangiansandspeciallagrangians}. $F$ turns out to be an elliptic operator with surjective linearization at the origin.
     
     \begin{proposition}
     \label{prop:HWB-ellipticity-of-SLAG-operator}
         Let $f_0:L \rightarrow X$ be smooth free special Lagrangian immersion with boundary conditions $\Lambda_1,\dots, \Lambda_d$ (Definition~\ref{def:HWB-boundary-conditions-on-immersions-and-lagrangians}). Let $f_\theta$ and $F$ be as in \eqref{eq:HWB-normal-field-to-deformation-formula-2} and \eqref{eq:HWB-linearization-of-SLAG-2} respectively. Then the following hold.
         \begin{enumerate}[label={(\roman*)}]
         \item $F(\theta) \in dA^1_D(L) \oplus dA^{n-1}(L)$ for all $\theta \in A^1_D(L)$ (Definition \ref{def:HWB-vector-spaces-of-forms-and-boundary-conditions}).
         \item The linearization of $F$ at the origin is
         \begin{equation}
         \label{eq:HWB-formula-derivative-SLAG-operator}
         d_0F(\theta) = d\theta \oplus d\star_0 \theta
         \end{equation}
         where $\star_0$ is the Hodge star operator for the pullback metric $f_0^*g$.
         \item $\ker d_0F$ is the space of harmonic 1-fields on $L$ (Definition \ref{def:harmonicfields}) with Dirichlet boundary conditions.
         \item By (i), the linearization $d_0F$ takes values in $dA^1_D(L) \oplus dA^{n-1}(L)$. The linearization is surjective onto this subspace as a map from $C^{k+1,\alpha}$ forms to $C^{k,\alpha}$ forms, i.e.,
         \[d_0F: C^{k+1,\alpha}\cap A^1_D(L) \rightarrow C^{k,\alpha} \cap dA^1_D(L) \oplus C^{k,\alpha} \cap dA^{n-1}(L)\]
         is surjective for all $k\ge 1$ and $\alpha \in (0,1)$.
         \end{enumerate}
     \end{proposition}
     \begin{proof}~
         \begin{enumerate}[label={(\roman*)}]
             \item Let $\beta$ be any closed differential form on $X$ such that $f_0^*\beta = 0$. Then by Cartan's formula for Lie derivatives,
             \[\begin{aligned}
                 f_{\theta}^*\beta =&\ f_0^*\beta + \int_{0}^1 \frac{d}{ds} (f_{s\theta}^*\beta) ds\\
                 =&\ \int_{0}^1 f_{s\theta}^*\left( di_{\frac{df_{s\theta}}{ds}} \beta + i_{\frac{df_{s\theta}}{ds}} d\beta \right) ds \\
                 =&\ d \left(\int_{0}^1 f_{s\theta}^* i_{\frac{df_{s\theta}}{ds}} \beta ds\right).
             \end{aligned}\]
             Applying this argument to $\beta = \omega, \Im(\Omega)$ shows that $f_\theta^* \omega$ and $f_{\theta}^*\Im(\Omega)$ are exact. Finally, observe that for $\beta = \omega$,
             \[\begin{aligned}
                 \left( f_{s\theta}^* i_{\frac{df_{s\theta}}{ds}} \omega \right)(Y) =&\ \left(f_{s\theta}^* i_{\frac{df_{s\theta}}{ds}} \omega \right)(Y) \\
                 =&\ \omega\left( \frac{df_{s\theta}}{ds}, df_{s\theta}(Y) \right) \qquad \forall\ Y\in T\partial L.
             \end{aligned}\]
             Fix $p\in C_i$ and recall that $f_{s\theta}(C_i) \subset \Lambda_i$ for all $i$. Then by construction of $f_\theta$, $\frac{df_{s\theta}}{ds}$ is tangent to $\Lambda_i$ and so is $df_{s\theta}(Y)$. Since $\Lambda_i$ is Lagrangian, the RHS above will vanish. Conclude that
             \[ \left(f_{s\theta}^* i_{\frac{df_{s\theta}}{ds}} \omega \right)(Y) = 0\qquad \forall\ Y\in T\partial L,\]
             \[\implies \quad \left(\int_{0}^1 f_{s\theta}^* i_{\frac{df_{s\theta}}{ds}} \omega ds\right) \in  A^1_D(L)\]
             \[\implies \quad f_\theta^*\omega = d\left(\int_{0}^1 f_{s\theta}^* i_{\frac{df_{s\theta}}{ds}} \omega ds\right) \in  dA^1_D(L).\]
            This completes the proof of (i).
    
             \item Computing the linearization of $F $ at the origin,
                 \[\begin{aligned}
                     d_0F(\theta) =&\ \left. \frac{d}{ds} F(s\theta) \right|_{s=0} = \left. \frac{d}{ds}  f_{s\theta}^*\omega \oplus f_{s\theta}^* \Im(\Omega) \right|_{s=0} \\
                     =&\ f_{0}^*\left( di_{\Phi^{-1}\theta} \omega +  i_{\Phi^{-1}\theta} d \omega \right) \oplus f_{0}^*\left( di_{\Phi^{-1}\theta} \Im(\Omega)  +  i_{\Phi^{-1}\theta} d \Im(\Omega) \right) \\
                     =&\ f_{0}^*\left( di_{\Phi^{-1}\theta} \omega \right) \oplus f_{0}^*\left( di_{\Phi^{-1}\theta} \Im(\Omega) \right),\\
                     =&\ d\theta \oplus df_{0}^*\left( i_{\Phi^{-1}\theta} \Im(\Omega) \right),
                 \end{aligned}\]
                 where the second line follows by Cartan's formula for Lie derivatives and the formula $\frac{df_{s\theta}}{ds}|_{s=0} = \Phi^{-1} \theta\circ f_0$, the second line by closedness of $\omega, \Im(\Omega)$, and the final line by the definition of $\Phi$. Finally by \eqref{eq:HWB-normal-field-to-deformation-formula-3} and Lemma~\ref{lemma:HWB-phitishodgestarofthetat},
                 \[f_{0}^*\left( i_{\Phi^{-1}\theta} \Im(\Omega) \right) = \star_0 f_{0}^*\left( i_{\Phi^{-1}\theta} \omega \right) = \star_0 \theta\]
                 since $\frac{df_{s\theta}}{ds}|_{s=0} = \Phi^{-1} \theta\circ f_0$, where $\star_0$ is the Hodge star operator for $f_0^*g$. So
                 \begin{equation*}
                     d_0F(\theta) = d\theta \oplus d\star_0 \theta.
                 \end{equation*}
            
            \item If $d_0F(\theta) = 0$ then $d\theta = 0$ and $d\star \theta = 0$. Thus $\theta$ is closed and co-closed making it a harmonic field. Also note that any element in the domain of $d_0F$ has Dirichlet boundary conditions.

            \item Pick an arbitrary element from the codomain of $d_0F$, i.e., pick
            \begin{equation}
                \label{eq:HWB-SLAG-def-arbitrary-element}
                \theta \in C^{k+1,\alpha} \cap A^1_D(L),\quad \theta|_{\partial L} \equiv 0, \quad \phi \in C^{k+1,\alpha} \cap A^{n-1}(L).
            \end{equation}
            Then $d\theta \oplus d\phi$ is an arbitrary element of the codomain of $d_0F$. The strategy will be pick find a $C^{k+2,\alpha}$ function $f$, vanishing on the boundary such that $d_0F(\theta + df) = d\theta \oplus d\phi$.
            \begin{claim}
            \label{claim:HWB-elliptic-regularity}
                Let $L$ be a smooth Riemannian manifold with boundary. Fix $\theta \in C^{k+1,\alpha} \cap A^1_D(L)$ and $\phi \in C^{k+1,\alpha} A^{n-1}(L)$. Then there exists $f\in C^{k+2, \alpha}$ satisfying
                \[d\star\theta + d\star d f = d\phi.\]
                \[\begin{aligned}
                    d\star\theta + d\star d f = d\phi, \qquad &\text{on } L,\\
                    f \equiv 0 \qquad &\text{on } \partial L.
                \end{aligned}\]
            \end{claim}
            \begin{proof}
                Apply Hodge star to the given equation to get
                \begin{equation}
                    \label{eq:HWB-hodge-elliptic-PDE}
                    \begin{aligned}
                    \Delta f = \star d\star df = \star d\phi - \star d\star \theta, \qquad &\text{on } L,\\
                    f = 0 \qquad &\text{on } \partial L.
                \end{aligned}
                \end{equation}
                The RHS of the equation above is in $C^{k, \alpha}$. By the existence and regularity of solutions to elliptic PDE (Taylor \cite[Proposition 1.1, Proposition 1.6]{Taylor2011}), the PDE \eqref{eq:HWB-hodge-elliptic-PDE} admits a unique solution in $C^{k+2,\alpha}$.
            \end{proof}

            Apply Claim~\ref{claim:HWB-elliptic-regularity} to the given pair $\theta, \phi$ satisfying \eqref{eq:HWB-SLAG-def-arbitrary-element}, to obtain $f\in C^{k+2,\alpha}(L)$, vanishing on the boundary such that
            \[d_0F(\theta + df) = d\theta \oplus d \star (\theta + df) =  d\theta \oplus d\phi.\]
            $\theta + df \in C^{k+1,\alpha} \cap A^1_D(L)$ by properties of $f$, and so $d_0F$ is surjective from $C^{k+1}$ forms to $C^{k}$ forms. \qedhere
         \end{enumerate}
     \end{proof}

     \begin{corollary}
         \label{corollary:HWB-lifted-local-coordinate-chart}
         Fix a free immersed special Lagrangian $\zcal_0 \in \scal\lcal(X,L; \Lambda_1,\dots, \Lambda_d)$ and a free special Lagrangian immersion $f_0:L \rightarrow X$ representing it. Fix a basis
         \[\theta_1,\dots \theta_m  \in \ker d_0 F = Z^1_D \cap cZ^1(L)\]
         of harmonic fields with Dirichlet boundary conditions on $L$, with respect to the pullback metric $f_0^*g$.  Let $t_1,\dots, t_m$ denote coordinates on $\RR^m$. There also a simply-connected domain $0\in U \subset \RR^m$ and a smooth map
         \[f: U\times L \rightarrow X \]
         satisfying the following:
         \begin{enumerate}[label={(\roman*)}]
             \item $f(t,\cdot):L \rightarrow X$ is a free special Lagrangian immersion with boundary in $\Lambda_1,\dots, \Lambda_d$ for all $t \in U$,
             \item $f(0,\cdot) = f_0 $, the given immersion representing $\zcal_0$, and
             \item  $\displaystyle{f_0^* \left( \left. i_{\partial_{i} f}  \omega \right|_{t=0} \right) = \theta_i} $ for $i =1,\dots, m$, where $\partial_i f := \frac{\partial f}{ \partial t_i}$.
         \end{enumerate}
         Consequently, the map $t \mapsto [f_t]$ defines a smooth diffeomorphism (Definition~\ref{def:smoothpathoflagrangians}) from $U\subset \RR^m$ to some neighborhood $\ucal \subset \scal\lcal (X,L; \Lambda_1, \dots, \Lambda_d)$ of $\zcal_0$ mapping $0\in U$ to $\zcal_0 \in \ucal$.
     \end{corollary}
     \begin{proof}
         By Proposition~\ref{prop:HWB-ellipticity-of-SLAG-operator}, $d_0F$ is surjective and has a finite-dimensional kernel. $F^{-1}(0)$ is a smooth manifold in a small neighborhood of the origin. Explicitly, the map $G(\theta) := (\theta , F(\theta))$ is a diffeomorphism for $\theta$ is a small neighborhood of the origin in $ \ker d_0F$. But by the 1-1 correspondence of Corollary \ref{cor:HWB-boundary-lagrangian-deformation-space}, every free smooth immersion with boundary condition $\Lambda_1, \dots, \Lambda_d$ in a neighborhood of $f_0$ is of the form $f_\theta$ for some 1-form $\theta$ on $L$ vanishing on the boundary of $L$, and a neighborhood of $\zcal_0$ in $\scal\lcal(X,L; \Lambda_1,\dots, \Lambda_d)$ is identified with a neighborhood of $F^{-1}(0)$. Consider the isomorphism 
         \[\RR^m \rightarrow \ker d_0F, \qquad t = (t_1,\dots, t_m) \mapsto \theta_t := \sum_{i=1}^n t_i \theta_i.\]
         Combining all of these observations, a neighborhood $U$ of $0 \in \RR^m$ is diffeomorphic to a neighborhood $\ucal$ of $\zcal_0 \in \scal\lcal (X,L; \Lambda_1, \dots, \Lambda_d)$ via the map 
         \[ U \ni t \mapsto \theta_t = \sum_{i=1}^n t_i \theta_i \mapsto G^{-1}(\theta_t, 0) \mapsto f_{\theta_t} \mapsto [ f_{\theta_t} ] \in \scal\lcal (X,L; \Lambda_1, \dots, \Lambda_d)\]
         Now define $f: U\times L \rightarrow X $ by
         \[f(t,\cdot) := f_{\theta_t} \qquad \forall\ t\in U. \]
         Then $f$ satisfies the required properties. $f(t,\cdot)$ is a free special Lagrangian immersion by definition and $f(0,\cdot) = f_0$. Finally by \eqref{eq:HWB-normal-field-to-deformation-formula-3},
         \begin{equation}
             \frac{\partial f}{\partial t_i} = \frac{\partial f_{\theta_{t}}}{\partial t_i} \quad \implies \quad f_0^* \left( i_{ \partial_i f |_{t=0} } \omega \right)  = \frac{\partial \theta_{t}}{\partial t_i} = \theta_i. \tag*{\qedhere}
         \end{equation}
     \end{proof}

     \subsection{Proof of Theorem~\ref{thm:HWB-tangentspacetospeciallagrangianswithboundary-2}}
     \label{subsec:HWB-proof-of-thm-2-1}

\begin{proof}[Proof of Theorem~\ref{thm:HWB-tangentspacetospeciallagrangianswithboundary-2}]
Fix a special Lagrangian immersion $f_0:L \rightarrow X$ representing a generic element $[f_0] \in \scal\lcal(X,L; \Lambda_1, \dots, \Lambda_d)$. By Corollary \ref{cor:HWB-boundary-lagrangian-deformation-space}.(ii), every free smooth immersion with boundary condition $\Lambda_1, \dots, \Lambda_d$ in a neighborhood of $f_0$ is of the form $f_\theta$ for some 1-form $\theta$ modulo the reparametrization action of $\Diff(L)$. And so a neighborhood of $[f_0]$ in $\scal\lcal(X,L; \Lambda_1,\dots, \Lambda_d)$ is identified with a neighborhood of $F^{-1}(0)$.

By Proposition~\ref{prop:HWB-ellipticity-of-SLAG-operator}, $d_0F$ is surjective and has a finite-dimensional kernel. By the implicit function theorem, $F^{-1}(0)$ is a smooth manifold in a small neighborhood of the origin, and so $\scal\lcal(X,L;\Lambda_1,\dots, \Lambda_d)$ is a smooth manifold in a small neighborhood of the given immersed special Lagrangian $[f_0]$. The tangent space to $\scal\lcal(X,L;\Lambda_1,\dots, \Lambda_d)$ at $[f_0]$ is identified with $\ker d_0F$, the space of harmonic 1-fields with Dirichlet boundary conditions.

By Corollary \ref{cor:harmonicfieldsvanishingontheboundaryandcohomology}, this space of tangent fields is naturally isomorphic to the relative cohomology $H^1(L, \partial L; \RR)$. Finally by the Poincar\'e--Lefschetz duality for relative cohomology \cite[Theorem 3.43]{Hatcher-AT-MR1867354} and isomorphisms between a vector space and its dual, 
\[H^1(L, \partial L; \RR) \cong H^{n-k}(L; \RR). \]
Conclude that $\scal\lcal(X,L; \Lambda_1,\dots, \Lambda_d)$ is a finite-dimensional smooth manifold, whose dimension is $m = \dim H^1(L, \partial L; \RR) = \dim H^{n-k}(L; \RR) $.
\end{proof}

\section{Generalizations}
\label{sec:HWB-generalizations}

This section discusses generalizations of Proposition~\ref{prop:HWB-tangentspacetospeciallagrangianswithboundary-1}, and Theorem~\ref{thm:HWB-tangentspacetospeciallagrangianswithboundary-2} beyond Ricci flat Calabi--Yau manifolds $(X,\omega, J ,\Omega)$ from Definition~\ref{def:HWB-calabi-Yau-manifold}.

\begin{definition}
    An \textit{almost Calabi--Yau manifold} is a K\"ahler manifold $(X, \omega, J)$ with a trivial canonical bundle $K_X$. Since $K_X = \Lambda^{(n,0)}T^*X$ is trivial, $X$ admits a a nowhere vanishing $(n,0)$-form $\Omega$. There exists a smooth function $\rho:X \rightarrow (0,\infty)$ such that
    \begin{equation}
        \label{eq:HWB-almost-Calabi-Yau-equation}
        (-1)^{\frac{n(n-1)}{2}} \left( \frac{\sqrt{-1}}{2}\right)^n \Omega \wedge \overline{\Omega} = \rho^2 \frac{1}{n!} \omega^n.
    \end{equation}
    Compare \eqref{eq:HWB-almost-Calabi-Yau-equation} with \eqref{eq:HWB-calabi-Yau-equation}. If $\omega$ is in addition Ricci flat, then $\rho\equiv 1$ and $(X,\omega, J, \Omega)$ would be a Calabi--Yau manifold.
\end{definition}

On an almost Calabi--Yau manifold, special Lagrangians are not minimal surfaces for the Riemannian metric $g$. It turns out that special Lagrangians are minimal surfaces for the conformal metric
\begin{equation}
    \label{eq:HWB-almost-CY-conformal-metric}
    \widetilde{g} = \rho^{\frac{-2}{n}} g.
\end{equation}
Consider the following generalization of Lemma~\ref{lemma:HWB-harveylawson-calibration-lemma} \cite[Proposition 4.5]{joyce-MR2167283}.

\begin{lemma}
    \label{lemma:HWB-harlaw-for-almost-CY}
    Let $(X,\omega, J, \Omega)$ be an $n$-dim almost Calabi--Yau manifold satisfying \eqref{eq:HWB-almost-Calabi-Yau-equation} for a smooth function $\rho:X \rightarrow (0,\infty)$. Let $L$ be $n$-dim real manifold, and $f:L \rightarrow X$ be a special Lagrangian immersion, i.e., satisfying $f^* \omega \equiv 0$ and $f^* \Im(\Omega) \equiv 0$. Then
    \[f^* \Re(\Omega) \vol_{f^* \widetilde{g}} \]
    where $\widetilde{g}$ is defined by \eqref{eq:HWB-almost-CY-conformal-metric}. In particular $f(L)$ is calibrated by $\Re(\Omega)$ and hence is a minimal submanifold of the Riemannian manifold $(X,\widetilde{g})$.
\end{lemma}

Proposition~\ref{prop:HWB-tangentspacetospeciallagrangianswithboundary-1} and Theorem~\ref{thm:HWB-tangentspacetospeciallagrangianswithboundary-2} generalize to almost Calabi--Yau manifolds in a straight forward manner after replacing the pullback Riemannian metric $f^*g$ by the conformal metric $f^* \widetilde{g}$ defined by \eqref{eq:HWB-almost-CY-conformal-metric}. The proofs follow the same structure while keeping track of additional multiplicative coefficients in terms of $\rho$. In the proof of Lemma~\ref{lemma:HWB-phitishodgestarofthetat}, specifically in the proof of  Claim~\ref{claim:HWB-phi-hodgestar-of-theta-claim2}, Lemma~\ref{lemma:HWB-harveylawson-calibration-lemma} is replaced by Lemma~\ref{lemma:HWB-harlaw-for-almost-CY}.

\begin{appendices}

\section{Proof of Corollary~\ref{cor:HWB-boundary-lagrangian-deformation-space}(ii)}
\label{subsec:HWB-appendix-proof-of-converse}

\begin{proof}[Proof of Corollary~\ref{cor:HWB-boundary-lagrangian-deformation-space}(ii)]
    Recall that $L$ is a compact manifold with boundary compoennts $C_1,\dots, C_d$, and $\Lambda_1,\dots, \Lambda_d \subset X$ are given Lagrangian submanifolds, and $f:L \rightarrow X$ is a smooth Lagrangian immersion with boundary conditions $\Lambda_1,\dots, \Lambda_d$ (Definition~\ref{def:HWB-boundary-conditions-on-immersions-and-lagrangians}). The goal is to prove that any immersion $f_1:L \rightarrow X$ with boundary conditions $\Lambda_1,\dots, \Lambda_d$ that is sufficiently close to $f$ in the $C^1$ topology can be reparametrized to be of the form $f_{\theta}$ for some 1-form $\theta$.
    
    Recall from Proposition~\ref{prop:HWB-boundary-Lagrangian-vector-bundle-connection} the vector subbundle $E \rightarrow L$ of $f^*TX \rightarrow L$ and the bundle isomorphism
    \[\Phi:E \xrightarrow[]{\sim} T^*L, \qquad \Phi(V) = f^*i_V\omega. \]
    Let $\nabla^\Lambda$ be as in Proposition~\ref{prop:HWB-boundary-Lagrangian-vector-bundle-connection-2}, and  $\exp^\Lambda$ denote its exponential map. Define
    \begin{equation}
        \label{eq:HWB-appendix-converse-Psi-def}
        \Psi:T^*L \rightarrow X, \qquad \Psi(p,\theta) = \exp^\Lambda_{f(p)}(\Phi^{-1} \theta) \qquad \theta\in T_p^*L.
    \end{equation}
    Since $f$ is an immersion, $\Psi$ is a free immersion on a neighborhood of the zero section.
    \begin{claim}~
    \label{claim:HWB-appendix-converse-claim1}
    There exists a tubular neighborhood of the zero section $\vcal \subset T^*L$ satisfying the following
    \begin{enumerate}[label={(\roman*)}]
        \item $\Psi:\vcal \rightarrow X$ is an immersion and hence a local diffeomorphism onto its image.
        \item $\Psi(\vcal \cap \nu_{C_i}) \subset \Lambda_i$, where for each boundary component $C_i \subset \partial L$, $\nu_{C_i}$ denotes the conormal bundle
        \[ \nu_{C_i} := \{(p,\theta) \in T^*L: p\in C_i\ \&\  \theta(T_pC_i) = 0 \}. \]
        After further shrinking $\vcal$ if necessary, $(\Pi|_{\vcal})^{-1}(\lambda_i) \subset \nu_{C_i}$ for $i=1,\dots, d$.
    \end{enumerate}
         
    \end{claim}
    \begin{proof} $\Psi$ is an immersion by the standard argument computing the linearization of the exponential map. Part (ii) follows from Proposition~\ref{prop:HWB-boundary-Lagrangian-vector-bundle-connection}.
    \begin{enumerate}[label={(\roman*)}]
        \item Fix a generic point on the zero section, say $(p,0) \in T^*L$. It is sufficient to prove that $d\Psi_{(p,0)}$ is injective on $T_{(p,0)}T^*L$ for each $p\in L$.
        
        Recall the natural decomposition $T_{(p,0)} T^*L = T_pL \oplus T^*_pL$. Since $\Psi(p,0) = f(p)$ for all $p\in L$, 
        \begin{equation}
           \label{eq:HWB-appendix-converse-claim1-eq1}
           d\Psi_{(p,0)}(Y) = df_p(Y) \in df(T_pL) \qquad \forall Y\in T_pL \subset T_{(p,0)}T^*L.
        \end{equation}
        And $df_p:T_pL \rightarrow T_pX$ is injective since $f$ is an immersion by assumption. Next differentiate $\Psi$ at $(p,0)$ along the fiber directions $E_p$. Fix $\theta \in T^*_pL$. Then since $\exp^\Lambda$ is a Riemannian exponential map,
        \begin{equation}
            \label{eq:HWB-appendix-converse-claim1-eq2}
            \begin{aligned}
            d\Psi_{(p,0)}(\theta) =&\ \left. \frac{d}{dt} \right|_{t=0} \Psi(p,t\theta) \\
            =&\ \left. \frac{d}{dt} \right|_{t=0} \exp_{f(p)}(t\Phi^{-1} \theta) \\
            =&\ \Phi^{-1}\theta \in T^*_pL.
        \end{aligned}
        \end{equation}
        Since $\Phi$ is a bundle isomorphism, $d\Psi_{(0,p)}$ is injective on $T^*_pL \subset T_{(0,p)}T^*L$. By \eqref{eq:HWB-appendix-converse-claim1-eq1} and \eqref{eq:HWB-appendix-converse-claim1-eq2}, $d\Psi_{(p,0)}$ is the sum of two injective maps $T_pL \oplus T^*_pL \rightarrow df(T_pL) + E_p$. Furthermore by Proposition~\ref{prop:HWB-boundary-Lagrangian-vector-bundle-connection}.1(a), $E_p$ is transverse to $df(T_pL)$ for each $p\in L$. Thus $d\Psi_{(p,0)}$ is injective.
        
        Since $p\in L$ was arbitrary, $d\Psi$ is injective along the zero section. By compactness of $L$, $\Psi$ is an immersion on a small neighborhood of the zero section in $T^*L$. Let $\vcal$ denote this neighborhood.

        \item By Proposition~\ref{prop:HWB-boundary-Lagrangian-vector-bundle-connection}.1(e),
        \[\Phi ^{-1} \theta \in T_{f(p)} \Lambda_i \cap E_p \qquad \forall\ (p,\theta) \in \nu_{C_i}.\]
        By Proposition~\ref{prop:HWB-boundary-Lagrangian-vector-bundle-connection-2}, $\Lambda_i$ is a geodesic submanifold for $\nabla^\Lambda$ and so
        \[\exp^\Lambda_{f(p)}( T_{ f(p ) } \Lambda_i ) \subset \Lambda_i, \qquad p\in C_i,\ i=1,\dots, d\]
        Conclude from \eqref{eq:HWB-appendix-converse-Psi-def} that
        \[\Psi(\nu_{C_i}) \subset \Lambda_i.\]
        But $\Psi|_{\vcal}$ is an immersion by (i), and so is $\Psi|_{\vcal \cap \nu_{C_i}}$ is an immersion onto its image in $\Lambda_i$. Furthermore \[\dim \vcal \cap \nu_{C_i} = \dim C_i + 1 = \dim \lambda_i.\]
        Thus $\Psi|_{\vcal \cap\nu_{C_i}}: \vcal \cap \nu_{C_i} \rightarrow \Psi(\vcal) \cap \Lambda_i$ is a local diffeomorphism onto its image for $i=1,\dots, d$. Shrink $\vcal$ to ensure $\vcal \cap \Lambda_i$  equals this image for $i=1,\dots, d$ to conclude $(\Psi|_{\vcal})^{-1}(\Lambda_i) \subset \nu_{C_i}$. \qedhere
        \end{enumerate}
        \end{proof}

    \begin{figure}
        \centering
        \begin{tikzcd}
            T^*L \supset \vcal \quad \qquad  \arrow[d, "\pi"] \arrow[dr, "\Psi"] & \\
             L \arrow[r, "f"] \arrow[r, bend right, "f_1"] \arrow[u, bend left, "\widetilde{f_1}"]& X
        \end{tikzcd}
        \caption{$f_1:L \rightarrow X$ is an immersion in the neighborhood $\ical \subset  C^\infty(L,X; \Lambda_1,\dots, \Lambda_d)$ of $f$, and $\widetilde{f_1}$ is a lifting of this immersion.}
        \label{fig:HWB-appendix-lifting-diagram}
    \end{figure}
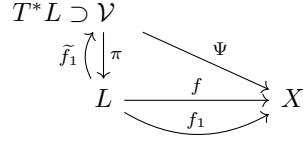
    Fix the neighborhood $\vcal \subset T^*L$ of the zero section from Claim~\ref{claim:HWB-appendix-converse-claim1}. The next step is to identify a $C^1$ neighborhood $\ical$ of $f$ in $C^\infty(L,X; \Lambda_1,\dots, \Lambda_d)$. If $\|f_1 - f\|_{C^1} <<1$ then by compactness of $L$ $f_1$ must also be an immersion, so it is sufficient to pick $\ical \subset \Imm(L,X) \cap C^\infty(L,X; \Lambda_1,\dots,\Lambda_d)$. Start by choosing a pair of open covers $\{W_\alpha: \alpha \in I\}$ and $\{ U_\alpha : \alpha \in I\}$ of $L$ satisfying the following.
    \begin{enumerate}
        \item $W_\alpha \subset \overline{W_\alpha} \subset U_\alpha \subset L$, and $W_\alpha$ is an open cover of $L$.
        \item By compactness of $L$, $\overline{W_\alpha}$ is compact. Furthermore assume that the indexing set $I$ is finite.
        \item $W_\alpha $ and $U_\alpha$ are either interior coordinate neighborhoods in $L$, or are boundary coordinate charts.
        \item $f|_{U_{\alpha}}$ is injective for each $\alpha \in I$, and after possibly shrinking $\vcal$, \[\Psi_\alpha := \Psi|_{\vcal \cap \pi^{-1} U_\alpha}: \vcal \cap \pi^{-1}U_\alpha \rightarrow X \]
        is a diffeomorphism for each $\alpha \in I$.
    \end{enumerate}
     Let $\ical \subset \Imm(L,X)\cap C^\infty (L,X; \Lambda_1,\dots, \Lambda_d)$ denote the open subset of immersions satisfying
    \begin{equation}
        \label{eq:HWB-appendix-converse-eq3}
        f_1(\overline{W_\alpha} \cap \overline{W}_\beta ) \subset \Psi(\vcal \cap \pi^{-1}U_\alpha \cap\pi^{-1}U_\beta) \qquad \forall\ \alpha, \beta \in I .
    \end{equation}
    Note that $\ical$ is an open subset of all $C^\infty(L,X; \Lambda_1,\dots, \Lambda_d)$ since \eqref{eq:HWB-appendix-converse-eq3} is an open condition, and $\ical$ contains a $C^0$ neighborhood of $f$ (and hence a $C^1$ neighborhood). 

    \begin{claim}
    \label{claim:HWB-appendix-converse-claim2}
        For each $f_1 \in \ical$ there exists a lifting $\widetilde{f_1}:L \rightarrow \vcal$ such that $\Psi\circ \widetilde{f_1} = f$ (see Figure~\ref{fig:HWB-appendix-lifting-diagram}). The lifting $\widetilde{f_1}$ is an immersions and $\widetilde{f_1}(C_i) \subset \nu_{C_i}$ for $i=1,\dots, d$.
    \end{claim}
    The zero section $L \rightarrow \vcal \subset T^*L$ is a lifting of the given immersion $f:L \rightarrow X$. Claim~\ref{claim:HWB-appendix-converse-claim2} identifies a section lifting any immersion in the neighborhood $\ical$ of $f$. The definition of $\ical$ from \eqref{eq:HWB-appendix-converse-eq3} differs slightly from that in Cervera--Mascar\'o--Michor \cite[Theorem 1.5]{Cervera-Mascaro-MichorMR1244452}. This choice is necessary for the lift \eqref{eq:HWB-appendix-claim2-lift-def} to be well defined. 
    \begin{proof}
         Recall the open cover $W_\alpha$ and $U_\alpha$ of $L$. Define $\widetilde{f_1}:L \rightarrow \vcal$ by
        \begin{equation}
            \label{eq:HWB-appendix-claim2-lift-def}
            \widetilde{f_1}(p):= \left( \Psi|_{\vcal \cap \pi^{-1} U_\alpha} \right)^{-1} (f_1(p)) \qquad \forall\ p\in W_\alpha.
        \end{equation}
        
        \begin{itemize}
        \item $\widetilde{f_1}$ is well defined: If $p \in W_\alpha \cap W_\beta$ then $f_1(p) \in \Psi(\vcal \cap \pi^{-1}(U_\alpha) \cap \pi^{-1}(U_\beta))$. But $\Psi|_{\vcal \cap \pi^{-1}U_\alpha \cap \pi^{-1} U_\beta}$ is a diffeomorphism, so
        \[\left( \Psi|_{\vcal \cap \pi^{-1} U_\alpha} \right)^{-1} (f_1(p)) = \left( \Psi|_{\vcal \cap \pi^{-1} U_\beta} \right)^{-1} (f_1(p)).\]
        \item $\widetilde{f_1}$ is an immersion: By \eqref{eq:HWB-appendix-claim2-lift-def}, $\widetilde{f_1}|_{W_{\alpha}}$ is the composition of an immersion and a diffeomorphism. Since $(W_{\alpha})_{\alpha \in I}$ covers $L$, $\widetilde{f_1}$ is an immersion.

        \item $\widetilde{f_1}(C_i) \subset \nu_{C_i}$: Note that $f_1(p) \in \Lambda_i$ for each $p\in C_i$ by assumption, and $\Psi^{-1}(\Lambda_i) \subset \nu_{C_1}$ by Claim~\ref{claim:HWB-appendix-converse-claim1}. So $\widetilde{f_1}(p) \in \nu_{C_i}$ for every $p\in C_i$, $i=1,\dots, d$. \qedhere
        \end{itemize}
    \end{proof}

    The final step is shrink the neighborhood $\ical$---defined by \eqref{eq:HWB-appendix-converse-eq3}---to ensure that the lifting constructed in Claim~\ref{claim:HWB-appendix-converse-claim2} can be reparametrized to be of the form $f_{\theta}$ for some 1-form $\theta$. The following lemma is a general statement of about openness of relative homotopy classes in the space of smooth maps.

    \begin{lemma}
    \label{lemma:HWB-appendix-homotopy-class-open}
        Let $f:(M,\partial M) \rightarrow (N, \partial N)$ be a smooth boundary preserving map of compact manifolds. Then the homotopy class of $f$ is an open subset of the space of smooth boundary preserving maps $C^\infty ((M,\partial M), (N, \partial N))$,
    \end{lemma}
    \begin{proof}
        Assume WLOG that $N\subset H:= \RR^K\times[0,\infty)$ for some large $K>>1$ such that $\partial N = \RR^K\times \{0\} \cap N $ \cite[Theorem~4.3]{Hirsch-diff-top-MR448362}. By the tubular neighborhood theorem for manifolds with boundary \cite[Theorem~6.3, 6.4]{Hirsch-diff-top-MR448362}, $N$ admits a tubular neighborhood $T$ in the half space $\RR^K\times[0,\infty)$. By compactness of $N$, there is some $\epsilon >0$ such that 
        \[\{y \in H: d(y, M)< \epsilon \} \subset T.\]
        Consider the neighborhood of $f$ 
        \[ \ucal := \{ f':(M, \partial M) \rightarrow (N, \partial N): \sup_{x\in M} d_{H}(f(x), f'(x) ) < \epsilon \}.\]
        For any $f'\in \ucal$, $f_t := tf + (1-t)f$ is a smooth homotopy of maps $(M,\partial M) \rightarrow (T,\partial T)$. But $T$ being a tubular neighborhood of $N$ deformation retracts onto $N$, and let $r:T \rightarrow N$ denote this retraction. Then $r\circ f_t : (M, \partial M) \rightarrow (N,\partial N)$ defines a smooth homotopy between $f$ and $f'$. Thus $\ucal$ is the required neighborhood.
    \end{proof}

    Corollary~\ref{cor:HWB-boundary-lagrangian-deformation-space}(ii) now follows from Claim~\ref{claim:HWB-appendix-converse-claim3}.

    \begin{claim}
    \label{claim:HWB-appendix-converse-claim3}
        For each $f_1 \in \ical$ there exists a smooth map $\varphi = \varphi_{f_1} :L \rightarrow L$ and a 1-form $\theta = \theta_{f_1}:L \rightarrow \vcal \subset T^*L$ such that
        \[\widetilde{f_1}(p) = (\varphi(p), \theta_{ \varphi(p)}).\]
        There is an open neighborhood $\ical' \subset \ical$ of $f$, such that $\varphi_{f_1}$ is a diffeomorphism for each $f_1 \in \ical'$. For any such $f_1\in \ical''$ (see \ref{eq:HWB-normal-field-to-deformation-formula-2}), 
        \[f_1 \circ \varphi^{-1}(p) = f_{\theta}.\]
    \end{claim}
    \begin{proof}
        Any element $f_1\in \ical$ is an immersion by the definition of $\ical$. Furthermore $\Psi$ is a local diffeomorphism. The existence of $\varphi_{f_1}$ and $\theta_{f_1}$ follows immediately.
        
        By Lemma~\ref{lemma:HWB-appendix-homotopy-class-open} applied to $f:(L,\partial L) \rightarrow (\overline{\vcal}, \partial \overline{\vcal})$, there is an open neighborhood $\ical'' \subset \ical$ containing $f$ such that for any $f_1\in \ical''$, its lifting $\widetilde{f_1}$ is homotopic to the zero section in $T^*L$ that lifts $f$. Composing this homotopy with the projection $\pi:T^*L \rightarrow L$ obtain a homotopy from $\varphi$ to $\id_L:L \rightarrow L$. There is a neighborhood of $\id_L$ in $C^\infty(L,L)$ in whih every element is a diffeomorphism on $L$. Pulling this neighborhood back to $C^\infty(L,X)$ under $\pi$, obtain a neighborhood $\ical'' \subset \ical$ such that $\pi\circ \widetilde{f_1} = \varphi_{f_1}$ is diffeomorphism for each $f_1\in \ical ''$. Finally for any such $f_1\in \ical''$, by \eqref{eq:HWB-normal-field-to-deformation-formula-2}
        \begin{equation}
            \widetilde{f_1}\circ \varphi^{-1}(p) = (p, \theta_p) \qquad \implies \qquad f-1 \circ \varphi^{-1} = f_{\theta}. \tag*{\qedhere}
        \end{equation}
    \end{proof}
\end{proof}

\end{appendices}

\printbibliography

\textsc{University of Maryland, College Park, Maryland}

\href{mailto:pvasanth@umd.edu}{pvasanth@umd.edu}

\end{document}